\documentclass[12pt,a4paper,reqno,twoside]{amsart}

\usepackage[english]{babel}
\usepackage{stmaryrd}
\usepackage{dsfont}
\usepackage[symbol*,ragged]{footmisc}
\usepackage[colorlinks,linkcolor=red,anchorcolor=blue,citecolor=blue,urlcolor=blue]{hyperref}
\usepackage{color,xcolor}

\usepackage{geometry}
\usepackage{amssymb}
\usepackage{amsmath}
\usepackage{mathrsfs}
\usepackage{amsfonts}
\usepackage{epsfig}

\usepackage{amsthm}
\usepackage{amsxtra}
\usepackage{bbding}
\usepackage{epsfig}
\usepackage{graphicx}
\usepackage{latexsym}
\usepackage{mathbbol}
\usepackage{bbold}

\usepackage{pifont}
\usepackage{wasysym}
\usepackage{skull}
\usepackage{float}

\DeclareSymbolFontAlphabet{\mathbb}{AMSb}
\DeclareSymbolFontAlphabet{\mathbbol}{bbold}

\usepackage{amscd}
\usepackage[all]{xy}
\allowdisplaybreaks[4]
\usepackage{setspace}

\theoremstyle{plain}
\newtheorem{theorem}{\normalfont\scshape Theorem}[section]
\newtheorem{proposition}{\normalfont\scshape Proposition}[section]
\newtheorem{lemma}[proposition]{\normalfont\scshape Lemma}
\newtheorem{corollary}[theorem]{\normalfont\scshape Corollary}
\newtheorem*{corollary*}{\normalfont\scshape Corollary}
\newtheorem{remark}{\normalfont\scshape Remark}

\theoremstyle{remark}
\newtheorem*{remark*}{\normalfont\scshape Remark}
\newtheorem*{notation}{\normalfont\scshape Notation}

\numberwithin{equation}{section}
\addtocounter{footnote}{1}

\renewcommand{\footnoterule}{
  \kern -3pt
  \hrule width 2.5in height 0.4pt
  \kern 3pt
}

\makeatletter
\@ifundefined{MakeUppercase}{}{}
\makeatother

\begin{document}
	
\title[On Chen's theorem over Piatetski--Shapiro type primes and almost--primes]
	  {On Chen's theorem over Piatetski--Shapiro type primes and almost--primes}

\author[Jinjiang Li, Fei Xue, Min Zhang]
       {Jinjiang Li \quad \& \quad Fei Xue \quad \& \quad Min Zhang }

\address{Department of Mathematics, China University of Mining and Technology,
         Beijing 100083, People's Republic of China}

\email{jinjiang.li.math@gmail.com}

%\author[Fei Xue]{Fei Xue}

\address{Department of Mathematics, China University of Mining and Technology,
         Beijing 100083, People's Republic of China}

\email{fei.xue.math@gmail.com}
	
%\author[Min Zhang]{Min Zhang}

\address{(Corresponding author) School of Applied Science, Beijing Information Science and Technology University,
		 Beijing 100192, People's Republic of China  }

\email{min.zhang.math@gmail.com}

\date{}

\footnotetext[1]{Min Zhang is the corresponding author. 

\quad\,\,
{\textbf{MR(2020) Subject Classification}}: 11L20, 11N05, 11N36, 11P32

}

%\footnotetext[3]{Corresponding author. \\
%    \quad\,\, \textit{ E-mail addresses}:
%     \href{mailto:min.zhang.math@gmail.com}{min.zhang.math@gmail.com} (M. Zhang),
%     \href{mailto:fei.xue.math@gmail.com}{fei.xue.math@gmail.com} (F. Xue), \\
%     \qquad\qquad\qquad\qquad\qquad\,\,\href{mailto:jinjiang.li.math@gmail.com}{jinjiang.li.math@gmail.com} (J. Li)     }

\begin{abstract}
In this paper, we establish a new mean value theorem of Bombieri--Vinogradov type over Piatetski--Shapiro sequence. Namely, it is proved that for any given constant $A>0$ and any sufficiently small $\varepsilon>0$, there holds
\begin{equation*}
   \sum_{\substack{d\leqslant x^\xi\\ (d,l)=1}}\Bigg|\sum_{\substack{A_1(x)\leqslant a<A_2(x)\\ (a,d)=1}}g(a)
   \Bigg(\sum_{\substack{ap\leqslant x\\ ap\equiv l\!\!\!\!\!\pmod d\\
          ap=[k^{1/\gamma}]}}1-\frac{1}{\varphi(d)}\sum_{\substack{ap\leqslant x\\ ap=[k^{1/\gamma}] }}
          1\Bigg)\Bigg|\ll\frac{x^\gamma}{(\log x)^A},
\end{equation*}
provided that $1\leqslant A_1(x)<A_2(x)\leqslant x^{1-\varepsilon}$ and $g(a)\ll \tau_r^s(a)$, where
$l\not=0$ is a fixed integer and
\begin{equation*}
\xi:=\xi(\gamma)=\frac{2^{38}+17}{38}\gamma-\frac{2^{38}-1}{38}-\varepsilon
\end{equation*}
with \begin{equation*}
  1-\frac{18}{2^{38}+17}<\gamma<1.
\end{equation*}
Moreover, for $\gamma$ satisfying 
\begin{equation*}
1-\frac{0.03208}{2^{38}+17}<\gamma<1,
\end{equation*}
we prove that there exist infinitely many primes $p$ such that $p+2=\mathcal{P}_2$ with $\mathcal{P}_2$ being Piatetski--Shapiro almost--primes of type $\gamma$, and there exist infinitely many Piatetski--Shapiro primes $p$ of type $\gamma$ such that $p+2=\mathcal{P}_2$.
These results generalize the result of Pan and Ding \cite{Pan-Ding-1979} and constitutes an improvement upon a series of previous results of \cite{Li-Zhang-Xue-2022,Lu-2018,Peneva-2003,Wang-Cai-2011}.
\end{abstract}

\maketitle

\section{Introduction and main result}
The ternary Goldbach problem asserts that every odd integer $n\geqslant9$ can be represented in the form
\begin{equation}\label{three-prime}
   n=p_1+p_2+p_3,
\end{equation}
where $p_1,p_2,p_3$ are odd prime numbers. In 1937, Vinogradov \cite{Vinogradov-1937} solved that asymptotic form of the representation of the type (\ref{three-prime}) exists for every sufficiently large odd integer. Recently, Helfgott \cite{Helfgott-1,Helfgott-2,Helfgott-3,Helfgott-4} completely solved the problem and proved that the ternary Goldbach conjecture is true. The binary Goldbach problem, which states that every even integer $N\geqslant6$ can be written as the sum of two odd primes, also remains unsettled. Another central problem in the theory of prime distribution, namely the twin prime conjecture, states that there exist infinitely many primes $p$ such that $p+2$ is also prime. Although the conjecture has resisted all attacks, there have been spectacular partial achievements. One of the well--known results is due to Chen \cite{Chen-1966,Chen-1973}, who proved that there exist infinitely many primes $p$ such that $p+2$ has at most $2$ prime factors.

An important approach for investigating the binary Goldbach problem is by the use of sieve methods, especially for Selberg's sieve method. As usual, we denote by $\mathcal{P}_r$ an almost--prime with at most $r$ prime factors, counted according to multiplicity. In Selberg's sieve method, the estimate of the error term is quite crucial. The mean value theorem can be used to achieve this estimate on average over special arithmetic progressions. Generally speaking, stronger results can be derived if a better improvement can be made on the extension of the `level' of the mean value theorem.

In 1947, by using the large sieve method developed by Linnik, R\'{e}nyi \cite{Renyi-1947} was the first to
prove that the following estimate
\begin{equation}\label{B-V-ini}
 \sum_{d\leqslant x^\xi}\max_{y\leqslant x}\max_{(l,d)=1}\Bigg|
 \sum_{\substack{n\leqslant y\\ n\equiv l\!\!\!\!\!\pmod d}}\Lambda(n)-\frac{y}{\varphi(d)}\Bigg|\ll
 \frac{x}{(\log x)^A}
\end{equation}
holds for $\xi<\xi_0$ with some fixed $\xi_0<1$, where $A>0$ is arbitrary. From (\ref{B-V-ini}), R\'{e}nyi \cite{Renyi-1947} showed that there exists an $r\in\mathbb{N}^+$ such that every sufficiently large even integer $N$ is representable in the form
\begin{equation}\label{binary-Goldbach}
   N=p+\mathcal{P}_r,
\end{equation}
where $p$ is a prime number. However, R\'{e}nyi did not give a quantitative connection between $\xi_0$ and $r$.
Later, Barban \cite{Barban-1961} and Pan \cite{Pan-1963-1,Pan-1964-1} proved that (\ref{B-V-ini}) holds for $\xi_0=1/6$
and $\xi_0=1/3$, independently and respectively. When $\xi_0=1/3$, Pan \cite{Pan-1963-1} firstly gave a quantitative estimate of $r$ and showed that (\ref{binary-Goldbach}) holds for $r=5$. Afterwards, Barban \cite{Barban-1963} and Pan \cite{Pan-1963-2}, independently and  respectively, proved that (\ref{B-V-ini}) holds for $\xi_0=3/8$, from which they both obtained that (\ref{binary-Goldbach}) holds for $r=4$. It is easy to see
that, under the generalized Riemann Hypothesis (GRH), (\ref{B-V-ini}) holds for $\xi_0=1/2$. In 1965, Bombieri
\cite{Bombieri-1965} and Vinogradov \cite{Vinogradov-1965} independently and unconditionally showed that (\ref{B-V-ini}) holds for
$\xi_0=1/2$. To be specific, Bombieri \cite{Bombieri-1965} proved that, for any given $A>0$, there holds
\begin{equation}\label{B-V-1965}
 \sum_{d\leqslant x^{1/2}(\log x)^{-B}}\max_{y\leqslant x}\max_{(l,d)=1}\Bigg|
 \sum_{\substack{n\leqslant y\\ n\equiv l\!\!\!\!\!\pmod d}}\Lambda(n)-\frac{y}{\varphi(d)}\Bigg|\ll
 \frac{x}{(\log x)^A},
\end{equation}
provided that $B\geqslant3A+23$. Afterwards, Gallagher \cite{Gallagher-1968} give a valuable simple proof of
(\ref{B-V-1965}). In 1965, by using the result of Barban \cite{Barban-1963}, i.e. $\xi_0=3/8$, Buchstab  \cite{Buchstab-1965} proved that (\ref{binary-Goldbach}) holds for $r=3$. The  hitherto best result in this direction is due to Chen \cite{Chen-1966,Chen-1973} who showed that (\ref{binary-Goldbach}) holds for $r=2$
by using the `weighted sieve' established by himself and the `switching principle' constructed by himself.
In his famous work, Chen \cite{Chen-1973} had to prove a result which is a varied theorem of Bombieri \cite{Bombieri-1965} on the distribution of primes in arithmetic progressions. Subsequently, Pan and Ding
\cite{Pan-Ding-1975-1,Pan-Ding-1975-2,Pan-Ding-1979} formulate a new mean value theorem which is a generalized
form of Bombieri--Vinogradov theorem that incorporates Chen's case. In \cite{Pan-Ding-1979}, Pan and Ding showed that, for any given $A>0$, if $1\leqslant A_1(x)<A_2(x)\leqslant x^{1-\varepsilon}$, there holds
\begin{equation}\label{Pan-Ding-mean}
\sum_{d\leqslant x^{1/2}(\log x)^{-B}}\max_{y\leqslant x}\max_{(l,d)=1}\Bigg|
\sum_{\substack{A_1(y)\leqslant a<A_2(y)\\ (a,d)=1}}f(a)
\Bigg(\sum_{\substack{an\leqslant y\\ an\equiv l\!\!\!\!\pmod d}}\Lambda(n)-\frac{y}{a\varphi(d)}\Bigg)
\Bigg|\ll\frac{x}{(\log x)^A},
\end{equation}
where $f(a)\ll\tau^k(a)$ and $B\geqslant2A+2^{2k+2}(2^{2k+2}+1)+21$, $\tau(a)$ denotes the Dirichlet divisor function. Here we emphasize that (\ref{B-V-1965}) can be deduced from (\ref{Pan-Ding-mean}) by choosing appropriate
function $f(a)$.

For $1/2<\gamma<1$, the Piatetski--Shapiro sequences are sequences of the form $([n^{1/\gamma}])_{n=1}^\infty$.
Such sequences have been named in honor of Piatetski--Shapiro, who, in \cite{Piatetski-Shapiro-1953},
proved that $([n^{1/\gamma}])_{n=1}^\infty$ contains infinitely many primes for $\frac{11}{12}<\gamma<1$. To be specific, Piatetski--Shapiro showed that, for $\gamma\in(\frac{11}{12},1)$, the counting function
\begin{equation*}
   \pi_\gamma(x):=\#\big\{p\leqslant x: p=[n^{1/\gamma}]\,\,\,\textrm{for some}\,\,\,n\in\mathbb{N}^+\big\}
\end{equation*}
satisfies the asymptotic formula
\begin{equation}\label{PS-asymp}
   \pi_\gamma(x)=\frac{x^\gamma}{\log x}(1+o(1))
\end{equation}
as $x\to\infty$. The range of $\gamma$ in the above asymptotic formula, in which it is
known that $([n^{1/\gamma}])_{n=1}^\infty$ contains infinitely many primes, has been extended many times
over the years due to a number of mathematicians \cite{BHR-1995,Heath-Brown-1983,Jia-1993,Jia-1994,Kolesnik-1972,Kolesnik-1985,Kumchev-1999,
Leitmann-1980,Liu-Rivat-1992,Rivat-1992}. The hitherto best result which makes (\ref{PS-asymp}) work is currently known to hold for all $\gamma\in(\frac{2426}{2817},1)$ thanks to Rivat and Sargos \cite{Rivat-Sargos-2001}. Moreover,
Rivat and Wu \cite{Rivat-Wu-2001} also showed that there exist infinitely many Piatetski--Shapiro primes for
$\gamma\in(\frac{205}{243},1)$ by showing a lower bound of $\pi_\gamma(x)$ with the excepted order of magnitude.

In 1992, Balog and Friedlander \cite{Balog-Friedlander-1992} firstly found an asymptotic formula for the number of solutions of the equation (\ref{three-prime}) with variables restricted to the Piatetski--Shapiro primes. An interesting corollary of their theorem is that every sufficiently large odd integer can be written as the sum of two primes and a Piatetski--Shapiro prime of type $\gamma$, provided that $\gamma\in(\frac{8}{9},1)$. Afterwards, their studies in this direction were subsequently continued by Jia \cite{Jia-1995} and by Kumchev \cite{Kumchev-1997}, and generalized by Cui \cite{Cui-2004} and Li and Zhang \cite{Li-Zhang-2018}, consecutively and respectively.

Based on the above results, it is interesting to investigate the solvability of the equation (\ref{binary-Goldbach}) when $p$ is a Piatetski--Shapiro prime. It is naturally expected that a theorem of Bombieri--Vinogradov type holds for the Piatetski--Shapiro primes. In the early days, the only result in this direction, due to Leitmann \cite{Leitmann-1977}, gives a very low level of distribution which does not allow us to determine the value of the parameter $r$.

In 2003, Peneva \cite{Peneva-2003} obtained a mean value theorem of Bombieri--Vinigradov's type for Piatetski--Shapiro primes. Namely, she proved that, for any given constant $A>0$ and any sufficiently small $\varepsilon>0$, there holds
\begin{equation}\label{Peneva-mean}
   \sum_{\substack{d\leqslant x^\xi\\ (d,l)=1}}\Bigg|\sum_{\substack{p\leqslant x\\ p\equiv l\!\!\!\!\!\pmod d\\
          p=[k^{1/\gamma}]}}1-\frac{1}{\varphi(d)}\pi_\gamma(x)\Bigg|\ll\frac{x^\gamma}{(\log x)^A},
\end{equation}
where $l\not=0$ is a fixed integer and
\begin{equation}\label{level-1}
  \xi=\xi(\gamma)=
  \begin{cases}
    \frac{755}{424}\gamma-\frac{331}{212}-\varepsilon, & \textrm{for
                 $\frac{662}{755}<\gamma\leqslant\frac{608}{675}$},\\
    \frac{5}{4}\gamma-\frac{13}{12}-\varepsilon, & \textrm{for $\frac{608}{675}<\gamma<1$}.
  \end{cases}
\end{equation}
By (\ref{level-1}) and sieve methods, Peneva \cite{Peneva-2003} firstly showed that for every sufficiently large even integer $N$, (\ref{binary-Goldbach}) is solvable with $p=[n^{1/\gamma}]$ a Piatetski--Shapiro prime, and $r$ being the least positive integer satisfying the inequality
\begin{equation*}
  r+1-\frac{\log\frac{4}{1+3^{-r}}}{\log 3}>\frac{1}{\xi(\gamma)}+\varepsilon.
\end{equation*}
Using the above level $\xi$ in (\ref{level-1}), Peneva \cite{Peneva-2003} proved that (\ref{binary-Goldbach}) is solvable for $r=7$ with a Piatetski--Shapiro prime $p=[n^{1/\gamma}]$ and $0.9854<\gamma<1$. Essentially, from the arguments similar to that in Peneva \cite{Peneva-2003}, one can obtain that, there exist infinitely many Piatetski--Shapiro primes of type $\gamma$ such that $p+2=\mathcal{P}_7$ with $0.9854<\gamma<1$.

In 2011, by using the same level $\xi$ in (\ref{level-1}), Wang and Cai \cite{Wang-Cai-2011} improved the result
of Peneva \cite{Peneva-2003}, and showed that there exist infinitely many Piatetski--Shapiro primes of type $\gamma$ such that $p+2=\mathcal{P}_5$ with $\frac{29}{30}<\gamma<1$. Afterwards, Lu \cite{Lu-2018}, in 2018, reestablished a mean value theorem of Bombieri--Vinigradov's type with level $\xi=\xi(\gamma)=(13\gamma-12)/4-\varepsilon$ for $\frac{12}{13}<\gamma<1$. By using this level, Lu \cite{Lu-2018} strengthened the result of Wang and Cai \cite{Wang-Cai-2011}. He proved that there exist infinitely many Piatetski--Shapiro primes of type $\gamma$ such that $p+2=\mathcal{P}_4$ with $0.9993<\gamma<1$. Subsequently,
Li, Zhang and Xue \cite{Li-Zhang-Xue-2022} reconstructed the mean value theorem of Bombieri--Vinigradov's type with level
\begin{equation*}
\xi=\xi(\gamma)=\frac{129}{4}\gamma-\frac{255}{8}-\varepsilon
\end{equation*}
for $\frac{85}{86}<\gamma<1$, by which and the weighted sieve of Richert they proved that there exist infinitely many Piatetski--Shapiro primes of type $\gamma$ such that $p+2=\mathcal{P}_3$ with $0.9989445<\gamma<1$.

In this paper, we shall firstly generalize the result of Pan and Ding \cite{Pan-Ding-1979} and
formulate a new mean value theorem which is a generalized form of Bombieri--Vinogradov's type over Piatetski--Shapiro sequences that incorporates Chen's argument \cite{Chen-1973}. After that, by using Chen's trick, i.e. `Chen's weighted sieve' and `switching principle', we shall establish stronger conclusion and improve the previous result about the topic of shifted prime on Piatetski--Shapiro sequence.

\begin{theorem}\label{Theorem-1}
  Suppose that $l\not=0$ is a fixed integer, and $\gamma$ is a real number satisfying
\begin{equation*}
  1-\frac{18}{2^{38}+17}<\gamma<1.
\end{equation*}
Then for any given constant $A>0$ and any sufficiently small $\varepsilon>0$, there holds
\begin{equation}\label{Thm-1-eq}
   \sum_{\substack{d\leqslant x^\xi\\ (d,l)=1}}\Bigg|\sum_{\substack{A_1(x)\leqslant a<A_2(x)\\ (a,d)=1}}g(a)
   \Bigg(\sum_{\substack{ap\leqslant x\\ ap\equiv l\!\!\!\!\!\pmod d\\
          ap=[k^{1/\gamma}]}}1-\frac{1}{\varphi(d)}\sum_{\substack{ap\leqslant x\\ ap=[k^{1/\gamma}] }}
          1\Bigg)\Bigg|\ll\frac{x^\gamma}{(\log x)^A},
\end{equation}
provided that $1\leqslant A_1(x)<A_2(x)\leqslant x^{1-\varepsilon}$ and $g(a)\ll \tau_r^s(a)$, where
\begin{equation*}
\xi:=\xi(\gamma)=\frac{2^{38}+17}{38}\gamma-\frac{2^{38}-1}{38}-\varepsilon;
\end{equation*}
the implied constant in (\ref{Thm-1-eq}) depends only on $A$ and $\varepsilon$.
\end{theorem}

In Theorem \ref{Theorem-1}, if we take $A_1(x)\equiv1$ and
\begin{equation}\label{g-spe}
g(a)=\begin{cases}
1, & \textrm{if}\,\,a=1,\\
0, & \textrm{if}\,\,a>1.
\end{cases}
\end{equation}
Then (\ref{Thm-1-eq}) becomes
\begin{equation*}
 \sum_{\substack{d\leqslant x^\xi\\ (d,l)=1}}\Bigg|
 \sum_{\substack{p\leqslant x\\ p\equiv l\!\!\!\!\!\pmod d\\ p=[k^{1/\gamma}]}}1
 -\frac{1}{\varphi(d)}\sum_{\substack{p\leqslant x\\ p=[k^{1/\gamma}] }}1
 \Bigg|\ll\frac{x^\gamma}{(\log x)^A},
\end{equation*}
so that Theorem \ref{Theorem-1} is a generalization of (\ref{Peneva-mean}) with its enlarged level $\xi\to\frac{9}{19}=0.473684\dots$ as $\gamma$ tends to $1$. Thus, one can obtain the following corollary.

\begin{corollary}\label{Coro-1}
 Suppose that $l\not=0$ is a fixed integer, and $\gamma$ is a real number satisfying
\begin{equation*}
  1-\frac{18}{2^{38}+17}<\gamma<1.
\end{equation*}
Then for any given constant $A>0$ and any sufficiently small $\varepsilon>0$, there holds
\begin{equation}\label{Coro-1-eq}
 \sum_{\substack{d\leqslant x^\xi\\ (d,l)=1}}\Bigg|
 \sum_{\substack{p\leqslant x\\ p\equiv l\!\!\!\!\!\pmod d\\ p=[k^{1/\gamma}]}}1
 -\frac{1}{\varphi(d)}\sum_{\substack{p\leqslant x\\ p=[k^{1/\gamma}] }}1
 \Bigg|\ll\frac{x^\gamma}{(\log x)^A},
\end{equation}
where
\begin{equation*}
\xi:=\xi(\gamma)=\frac{2^{38}+17}{38}\gamma-\frac{2^{38}-1}{38}-\varepsilon;
\end{equation*}
the implied constant in (\ref{Coro-1-eq}) depends only on $A$ and $\varepsilon$.
\end{corollary}

\begin{remark}
Under exponent pair hypothesis, i.e. $(\varepsilon,\frac{1}{2}+\varepsilon)$ is an exponent pair, it follows from Lemma \ref{Type-II-es}, Lemma \ref{Type-I-es} and Lemma \ref{exponen-fenjie} that
(\ref{Thm-1-eq}) holds for
\begin{equation}\label{EPH-level}
\xi=\xi(\gamma)=\frac{5}{2}\gamma-2-\varepsilon
\end{equation}
with $\frac{4}{5}<\gamma<1$, from which and (\ref{g-spe}) one can obtain (\ref{Peneva-mean}) with `level' defined as in (\ref{EPH-level}) provided that $\frac{4}{5}<\gamma<1$.
\end{remark}

\begin{remark}
 By using the method exactly the same as that during proving Theorem \ref{Theorem-1}, one can establish the following mean value theorem with Piatetski--Shapiro type's shifted product $ap$ as follows
\begin{equation}\label{Thm-1-eq-shifted}
   \sum_{\substack{d\leqslant x^\xi\\ (d,l)=1}}\Bigg|\sum_{\substack{A_1(x)\leqslant a<A_2(x)\\ (a,d)=1}}g(a)
   \Bigg(\sum_{\substack{ap\leqslant x\\ ap\equiv l\!\!\!\!\!\pmod d\\ ap\pm2=[k^{1/\gamma}]}}1
          -\frac{1}{\varphi(d)}\sum_{\substack{ap\leqslant x\\ ap\pm2=[k^{1/\gamma}] }}
          1\Bigg)\Bigg|\ll\frac{x^\gamma}{(\log x)^A},
\end{equation}
where the explicit formula of $\xi=\xi(\gamma)$ and the range of $\gamma$, which makes the (\ref{Thm-1-eq-shifted}) work, are exactly the same as is shown in Theorem \ref{Theorem-1}.
\end{remark}

\begin{theorem}\label{Theorem-2}
Suppose that $\gamma$ is a real number satisfying
\begin{equation*}
1-\frac{0.03208}{2^{38}+17}<\gamma<1.
\end{equation*}
Then there exist infinitely many primes $p$ such that
\begin{equation*}
p+2=\mathcal{P}_2
\end{equation*}
with $\mathcal{P}_2$ being Piatetski--Shapiro almost--primes of type $\gamma$.
\end{theorem}

\begin{theorem}\label{Theorem-3}
Suppose that $\gamma$ is a real number satisfying
\begin{equation*}
1-\frac{0.03208}{2^{38}+17}<\gamma<1.
\end{equation*}
Then there exist infinitely many Piatetski--Shapiro primes $p$ of type $\gamma$ such that
\begin{equation*}
   p+2=\mathcal{P}_2.
\end{equation*}
\end{theorem}

\begin{remark}
 The process of the proof of Theorem \ref{Theorem-3} is essentially the same as that of Theorem \ref{Theorem-2} except replacing Theorem \ref{Theorem-1} with Corollary \ref{Coro-1}.  Therefore, we only give the details of the proof of Theorem \ref{Theorem-2} and omit the proof of Theorem \ref{Theorem-3}.
\end{remark}

\begin{remark}
 The key point of improving the number $r$ such that $p+2=\mathcal{P}_r$ with Piatetski--Shapiro prime or almost--prime is to enlarge the level $\xi=\xi(\gamma)$, for $\gamma$ near to $1$, of the mean value theorem of Bombieri--Vinigradov's type over Piatetski--Shapiro sequence. In order to compare the result in this paper with the previous results (e.g., see the literatures \cite{Peneva-2003,Lu-2018,Li-Zhang-Xue-2022}), we list the numerical result as follows:
\begin{align*}
  \xi(\gamma)=&\,\, \frac{5}{4}\gamma-\frac{13}{12}-\varepsilon\to\frac{1}{6}=0.1666\dots, \quad\textrm{for $\gamma\to1$};    \\
  \xi(\gamma)=&\,\, \frac{13\gamma-12}{4}-\varepsilon\to \frac{1}{4}=0.25, \qquad\quad\,\, \textrm{for $\gamma\to1$};   \\
  \xi(\gamma)=&\,\, \frac{129}{4}\gamma-\frac{255}{8}-\varepsilon\to\frac{3}{8}=0.375, \quad\,\,\,\textrm{for $\gamma\to1$};   \\
  \xi(\gamma)=&\,\, \frac{2^{38}+17}{38}\gamma-\frac{2^{38}-1}{38}-\varepsilon\to \frac{9}{19}=0.473684\dots,
  \quad\textrm{for $\gamma\to1$};   \\
  \xi(\gamma)=&\,\, \frac{5}{2}\gamma-2-\varepsilon\to\frac{1}{2}=0.5,\quad\,\,\,
  \textrm{for $\gamma\to1$. \quad (under exponent pair hypothesis)}
\end{align*}
\end{remark}

\begin{notation}
 Throughout this paper, $x$ is a sufficiently large number; $\varepsilon$ and $\eta$ are sufficiently small positive numbers, which may be different in each occurrences. Let $p$ and $q$, with or without subscripts, always denote a prime number. We use $[x],\,\{x\}$ and $\|x\|$ to denote the integral part of $x$, the fractional part of $x$ and the distance from $x$ to the nearest integer, respectively.
Denote by $\mathcal{P}_r$ an almost--prime with at most $r$ prime factors, counted according to multiplicity.
As usual, $\varphi(n),\Lambda(n),\tau_r(n)$ and $\mu(n)$ denote Euler's function, von Mangoldt's function, the Dirichlet divisor function of dimension $r$, and M\"{o}bius' function, respectively. Let $(m_1,m_2,\dots,m_k)$ and $[m_1,m_2,\dots,m_k]$ be the greatest common divisor and the least common multiple of $m_1,m_2,\dots,m_k$, respectively. We write $L=\log x$; $e(t)=\exp(2\pi it)$; $\psi(t)=t-[t]-\frac{1}{2}$. The notation $n\sim X$ means that $n$ runs through a subinterval
of $(X,2X]$, whose endpoints are not necessarily the same in the different occurrences and may depend on the outer summation variables. $f(x)\ll g(x)$ means that $f(x)=O(g(x))$; $f(x)\asymp g(x)$ means that $f(x)\ll g(x)\ll f(x)$.
\end{notation}

\section{Preliminaries}\label{Preliminary}
In this section, we shall reduce the problem of estimating the sum in (\ref{Thm-1-eq}) to estimating exponential sums over primes.

For $1/2<\gamma<1$, it is easy to see that
\begin{equation*}
  [-k^\gamma]-[-(k+1)^\gamma]=
  \begin{cases}
    1, & \textrm{if\,\,$k=[m^{1/\gamma}]$\,\, for some $m\in\mathbb{N}^+$},\\
    0, & \textrm{otherwise}.
  \end{cases}
\end{equation*}
For convenience, we put $D=x^\xi$. In order to prove (\ref{Thm-1-eq}), we only need to show that
\begin{align*}
 & \,\,\sum_{\substack{d\leqslant D\\ (d,l)=1}}\Bigg|\sum_{\substack{A_1(x)\leqslant a<A_2(x)\\ (a,d)=1}}g(a)
       \Bigg(\sum_{\substack{an\leqslant x\\ an\equiv l\!\!\!\!\!\pmod d}}\Lambda(n)
       \big([-(an)^\gamma]-[-(an+1)^\gamma]\big)
                \nonumber \\
 & \,\, \qquad \qquad\qquad\qquad \qquad\qquad -\frac{1}{\varphi(d)}\sum_{an\leqslant x}
       \Lambda(n)\big([-(an)^\gamma]-[-(an+1)^\gamma]\big)
          \Bigg)\Bigg|\ll\frac{x^\gamma}{(\log x)^A}.
\end{align*}
Therefore, it is sufficient to prove that
\begin{align}\label{suffi-1}
 & \,\, \sum_{\substack{d\leqslant D\\ (d,l)=1}}\Bigg|\sum_{\substack{A_1(x)\leqslant a<A_2(x)\\ (a,d)=1}}g(a)
        \Bigg(\sum_{\substack{an\leqslant x\\ an\equiv l\!\!\!\!\!\pmod d}}\Lambda(n)
        \big((an+1)^\gamma-(an)^\gamma\big)
                \nonumber \\
 & \,\, \qquad\qquad\qquad\qquad\qquad\qquad -\frac{1}{\varphi(d)}\sum_{an\leqslant x}
        \Lambda(n)\big((an+1)^\gamma-(an)^\gamma\big)\Bigg)\Bigg|\ll\frac{x^\gamma}{(\log x)^A},
\end{align}
\begin{align}\label{suffi-2}
\sum_{\substack{d\leqslant D\\ (d,l)=1}}\Bigg|\sum_{\substack{A_1(x)\leqslant a<A_2(x)\\ (a,d)=1}}g(a)
\sum_{\substack{an\leqslant x\\ an\equiv l\!\!\!\!\!\pmod d}}\Lambda(n)
\big(\psi\big(-(an)^\gamma\big)-\psi\big(-(an+1)^\gamma\big)\big)\Bigg|\ll\frac{x^\gamma}{(\log x)^A},
\end{align}
and
\begin{align}\label{suffi-3}
\sum_{\substack{d\leqslant D\\ (d,l)=1}}\frac{1}{\varphi(d)}
\Bigg|\sum_{\substack{A_1(x)\leqslant a<A_2(x)\\ (a,d)=1}}g(a)\sum_{an\leqslant x}\Lambda(n)
\big(\psi\big(-(an)^\gamma\big)-\psi\big(-(an+1)^\gamma\big)\big)\Bigg|\ll\frac{x^\gamma}{(\log x)^A}.
\end{align}
The estimate (\ref{suffi-1}) can be obtained from the new mean value theorem of Bombieri--Vinogradov type, which was established by Pan and Ding \cite{Pan-Ding-1975-1,Pan-Ding-1975-2,Pan-Ding-1979}, by using partial summation to get rid of the smooth weighted function $f(t)=(t+1)^\gamma-t^\gamma$ and it holds for every $\gamma\in(1/2,1)$ and $D=x^{1/2-\varepsilon}$, where $\varepsilon>0$ is sufficiently small. Hence, we only have to show (\ref{suffi-2})
and (\ref{suffi-3}). First, we shall give the proof of (\ref{suffi-2}) in details, and (\ref{suffi-3}) can be treated similarly by following the processes of the proof of (\ref{suffi-2}). Clearly, (\ref{suffi-2}) will follow, if we can prove that for $X\leqslant x$, there holds
\begin{equation}\label{suffi-4}
\sum_{\substack{d\leqslant D\\ (d,l)=1}}\Bigg|\sum_{\substack{A_1(x)\leqslant a<A_2(x)\\ (a,d)=1}}g(a)
\sum_{\substack{an\sim X\\ an\equiv l\!\!\!\!\!\pmod d}}\Lambda(n)
\big(\psi\big(-(an)^\gamma\big)-\psi\big(-(an+1)^\gamma\big)\big)\Bigg|\ll\frac{x^\gamma}{(\log x)^A},
\end{equation}
Let $\eta>0$ be a sufficiently small number. If $X\leqslant x^{1-\eta}$, then the left--hand side of (\ref{suffi-4}) is
\begin{align*}
\ll & \,\, \sum_{\substack{d\leqslant D\\ (d,l)=1}}\Bigg|\sum_{\substack{A_1(x)\leqslant a<A_2(x)\\ (a,d)=1}}g(a)
\sum_{\substack{an\sim X\\ an\equiv l\!\!\!\!\!\pmod d}}\Lambda(n)
           \big((an+1)^\gamma-(an)^\gamma\big)\Bigg|
                   \nonumber \\
   & \,\,  +\sum_{\substack{d\leqslant D\\ (d,l)=1}}
           \Bigg|\sum_{\substack{A_1(x)\leqslant a<A_2(x)\\ (a,d)=1}}g(a)
           \sum_{\substack{an\sim X\\ an\equiv l\!\!\!\!\!\pmod d}}
           \Lambda(n)\big([-(an)^\gamma]-[-(an+1)^\gamma]\big)\Bigg|
                     \nonumber \\
\ll & \,\, L\sum_{A_1(x)\leqslant a<A_2(x)}\tau_r^s(a)\sum_{an\sim X}(an)^{\gamma-1}\tau(|an-l|)+L
           \sum_{A_1(x)\leqslant a<A_2(x)}\tau_r^s(a)
           \sum_{\substack{an\sim X\\ an=[k^{1/\gamma}]}}\tau(|an-l|)
                     \nonumber \\
\ll & \,\, X^{\frac{\eta}{3}}\sum_{b\sim X}b^{\gamma-1}\tau(b)+X^{\frac{\eta}{3}}
           \sum_{\substack{b\sim X\\ b=[k^{1/\gamma}]}}\tau(b)\ll X^{\gamma+\frac{\eta}{2}}\ll x^\gamma L^{-A}.
\end{align*}
Hence, we can assume that $x^{1-\eta}\leqslant X\leqslant x$. It is easy to see that, for $\xi\leqslant(1-\eta)/2$, there holds
\begin{equation*}
 X^\xi\leqslant D\leqslant X^{\xi+\frac{\eta}{2}}.
\end{equation*}
Now, we use the well--known expansions (e.g., see the arguments on page 245 of Heath--Brown \cite{Heath-Brown-1983})
\begin{equation}\label{psi-expan}
\psi(t)=-\sum_{0<|h|\leqslant H}\frac{e(th)}{2\pi i h}+O(E(t,H)),
\end{equation}
where
\begin{equation*}
E(t,H):=\min\bigg(1,\frac{1}{H\|t\|}\bigg)=\sum_{h=-\infty}^{\infty}b_he(th)
\end{equation*}
and
\begin{equation*}
b_h\ll\min\bigg(\frac{\log2H}{H},\frac{1}{|h|},\frac{H}{|h|^2}\bigg).
\end{equation*}
Inserting (\ref{psi-expan}) into the left--hand side of (\ref{suffi-4}), the contribution of the error term in   (\ref{psi-expan}) to the left--hand side of (\ref{suffi-4}) is
\begin{equation}\label{tri-fenjie}
\sum_{\substack{d\leqslant D\\ (d,l)=1}}\sum_{\substack{A_1(x)\leqslant a<A_2(x)\\ (a,d)=1}}g(a)
\sum_{\substack{an\sim X\\ an\equiv l\!\!\!\!\!\pmod d}}\Lambda(n)
\big(E((an)^\gamma,H)+E((an+1)^\gamma,H)\big)=E_1+E_2,
\end{equation}
say. We treat $E_1$ only, since the estimate of $E_2$ is exactly the same. For $E_1$, we have
\begin{align}\label{E_1-upper}
            E_1
 \ll & \,\, L\sum_{\substack{d\leqslant D\\ (d,l)=1}}
            \sum_{\substack{A_1(x)\leqslant a<A_2(x)\\ (a,d)=1}}\tau_r^s(a)
            \sum_{\substack{an\sim X\\ an\equiv l\!\!\!\!\!\pmod d}}E((an)^\gamma,H)
                  \nonumber \\
 \ll & \,\, X^{\eta/10}\sum_{\substack{d\leqslant D\\ (d,l)=1}}
            \sum_{\substack{m\sim X\\ m\equiv l\!\!\!\!\!\pmod d}}
            \Bigg(\sum_{\substack{m=an\\ A_1(x)\leqslant a<A_2(x)\\ (a,d)=1}}\Bigg)E(m^\gamma,H)
                  \nonumber \\
 \ll & \,\, X^{\eta/9}\sum_{\substack{d\leqslant D\\ (d,l)=1}}
            \sum_{\substack{m\sim X\\ m\equiv l\!\!\!\!\!\pmod d}}
            \sum_{h=-\infty}^\infty b_he(hm^\gamma)
                  \nonumber \\
 \ll & \,\, X^{\eta/9}\sum_{\substack{d\leqslant D\\ (d,l)=1}}\sum_{h=-\infty}^\infty|b_h|
            \Bigg|\sum_{\substack{m\sim X\\ m\equiv l\!\!\!\!\!\pmod d}}e(hm^\gamma)\Bigg|.
\end{align}
Now, we use the following estimate, which is an analogue of Lemma 1 of Heath--Brown \cite{Heath-Brown-1983} for arithmetic progressions, to give the upper bound estimate of the innermost sum in (\ref{E_1-upper}).
\begin{lemma}\label{ex-arith-pair}
  Let $1\leqslant d\leqslant X,\, X<X_1\leqslant 2X$. Then
\begin{equation*}
\sum_{\substack{X<m\leqslant X_1\\ m\equiv l\!\!\!\!\!\pmod d}}e(hm^\gamma)
\ll \min\Big(Xd^{-1}, d^{-1}|h|^{-1}X^{1-\gamma}+d^{\kappa-\ell}|h|^\kappa X^{\kappa\gamma-\kappa+\ell}\Big),
\end{equation*}
where $(\kappa,\ell)$ is an exponent pair.
\end{lemma}
\begin{proof}
See Lemma 2.1 of Li, Zhang and Xue \cite{Li-Zhang-Xue-2022}.
\end{proof}

Taking $(\kappa,\ell)=(\frac{1}{2},\frac{1}{2})$ in Lemma \ref{ex-arith-pair}, we obtain
\begin{align*}
  E_1 \ll & \,\, X^{\eta/9}\sum_{d\leqslant D}
                 \bigg(|b_0|Xd^{-1}+\sum_{h\not=0}|b_h|\Big(|h|^{-1}X^{1-\gamma}d^{-1}+
                 |h|^{1/2}X^{\gamma/2}\Big)\bigg)
                         \nonumber \\
  \ll & \,\,X^{1+\eta/8}H^{-1}+X^{1-\gamma+\eta/9}\sum_{d\leqslant D}d^{-1}\sum_{h\not=0}|h|^{-2}
                         \nonumber \\
      & \,\,+X^{\gamma/2+\eta/9}D\Bigg(\sum_{0<|h|\leqslant H}|h|^{-1/2}+H\sum_{|h|>H}|h|^{-3/2}\Bigg)
                           \nonumber \\
 \ll & \,\, X^{1+\eta/8}H^{-1}+X^{1-\gamma+\eta/8}+X^{\gamma/2+\eta/8}H^{1/2}D\ll x^\gamma L^{-A},
\end{align*}
provided that
\begin{equation}\label{suffi-condi-1}
 H=X^{1-\gamma+\eta}\qquad \textrm{and} \qquad \gamma>\frac{1}{2}+\xi.
\end{equation}
Therefore, it remains to show that
\begin{align}\label{suffi-S-condi}
 S:= & \,\, \sum_{\substack{d\leqslant D\\ (d,l)=1}}\sum_{0<h\leqslant H}\frac{1}{h}
            \sum_{\substack{A_1(x)\leqslant a<A_2(x)\\ (a,d)=1}}\tau_r^s(a)
                  \nonumber \\
 & \,\, \qquad \times\Bigg|\sum_{\substack{an\sim X\\ an\equiv l\!\!\!\!\!\pmod d}}
 \Lambda(n)\Big(e\big(-h(an)^\gamma\big)-e\big(-h(an+1)^\gamma\big)\Big)\Bigg|\ll x^\gamma L^{-A}.
\end{align}
Define
\begin{equation*}
   \phi_{h,a}(n)=1-e\big(h((an)^\gamma-(an+1)^\gamma)\big).
\end{equation*}
By partial summation, the innermost sum on the left--hand side of (\ref{suffi-S-condi}) is
\begin{align}\label{inner-sum-upper}
     & \,\, \sum_{\substack{an\sim X\\ an\equiv l\!\!\!\!\!\pmod d}}\Lambda(n)e(-h(an)^\gamma)\phi_{h,a}(n)
                \nonumber \\
   = & \,\, \int_{\frac{X}{a}}^{\frac{2X}{a}}\phi_{h,a}(t)\mathrm{d}
            \Bigg(\sum_{\substack{X/a<n\leqslant t \\ an\equiv l\!\!\!\!\!\pmod d}}
            \Lambda(n)e(-h(an)^\gamma)\Bigg)
                \nonumber \\
 \ll &\,\, \bigg|\phi_{h,a}\bigg(\frac{2X}{a}\bigg)\bigg|
           \Bigg|\sum_{\substack{an\sim X\\ an\equiv l\!\!\!\!\!\pmod d}}\Lambda(n)e(-h(an)^\gamma)\Bigg|
                \nonumber \\
     &\,\, +\int_{\frac{X}{a}}^{\frac{2X}{a}}
            \Bigg|\sum_{\substack{X/a<n\leqslant t \\ an\equiv l\!\!\!\!\!\pmod d}}
            \Lambda(n)e(-h(an)^\gamma)\Bigg|\bigg|\frac{\partial\phi_{h,a}(t)}{\partial t}\bigg|\mathrm{d}t
                \nonumber \\
 \ll &\,\, hX^{\gamma-1}\times\max_{X/a<t\leqslant2X/a}
           \Bigg|\sum_{\substack{X/a<n\leqslant t \\ an\equiv l\!\!\!\!\!\pmod d}}
           \Lambda(n)e(-h(an)^\gamma)\Bigg|,
\end{align}
where we use the estimate
\begin{equation*}
   \phi_{h,a}(t)\ll ha^{\gamma-1}t^{\gamma-1}\qquad \textrm{and}\qquad\,
   \frac{\partial\phi_{h,a}(t)}{\partial t}\ll ha^{\gamma-1}t^{\gamma-2}.
\end{equation*}
Inserting (\ref{inner-sum-upper}) into the left--hand side of (\ref{suffi-S-condi}), we obtain
\begin{align*}
  S \ll & \,\, X^{\gamma-1}\times\sum_{\substack{d\leqslant D\\ (d,l)=1}}\sum_{0<h\leqslant H}
               \sum_{\substack{A_1(x)\leqslant a<A_2(x)\\ (a,d)=1}}\tau_r^s(a)
               \Bigg|\sum_{\substack{an\sim X\\ an\equiv l\!\!\!\!\!\pmod d}}\Lambda(n)e(-h(an)^\gamma)\Bigg|
                    \nonumber \\
      = & \,\, X^{\gamma-1}\times\sum_{\substack{d\leqslant D\\ (d,l)=1}}\sum_{0<h\leqslant H}
               \sum_{\substack{A_1(x)\leqslant a<A_2(x)\\ (a,d)=1}}\tau_r^s(a)\cdot c(d,h,a)
           \sum_{\substack{an\sim X\\ an\equiv l\!\!\!\!\!\pmod d}}\Lambda(n)e(-h(an)^\gamma)
                    \nonumber \\
    \ll & \,\, X^{\gamma-1}\times\sum_{A_1(x)\leqslant a<A_2(x)}\tau_r^s(a)
               \sum_{an\sim X}\Lambda(n)\sum_{0<h\leqslant H}e(-h(an)^\gamma)
               \sum_{\substack{d\leqslant D\\ (d,l)=1\\ d|an-l}}c(d,h,a)
                    \nonumber \\
      = & \,\, X^{\gamma-1}\times\sum_{A_1(x)\leqslant a<A_2(x)}\tau_r^s(a)\sum_{an\sim X}\Lambda(n)G_a(n),
\end{align*}
where
\begin{equation*}
  G_a(n)=\sum_{0<h\leqslant H}\Xi_{h,a}(n)e(-h(an)^\gamma)
\end{equation*}
and
\begin{equation*}
  \Xi_{h,a}(n)=\sum_{\substack{d\leqslant D\\ (d,l)=1\\ d|an-l}}c(d,h,a), \qquad |c(d,h,a)|=1.
\end{equation*}
Accordingly, in order to establish the estimate (\ref{suffi-S-condi}), it is sufficient to show that
\begin{equation}\label{suffi-genera}
  \Bigg|\sum_{A_1(x)\leqslant a<A_2(x)}\tau_r^s(a)\sum_{an\sim X}\Lambda(n)G_a(n)\Bigg|\ll XL^{-A}.
\end{equation}
A special case of the identity of Heath--Brown \cite{Heath-Brown-1982} is given by
\begin{equation*}
 -\frac{\zeta'}{\zeta}=-\frac{\zeta'}{\zeta}(1-Z\zeta)^3-\sum_{j=1}^3\binom{3}{j}(-1)^jZ^j\zeta^{j-1}(-\zeta'),
\end{equation*}
where $Z=Z(s)=\sum\limits_{k\leqslant (X/a)^{1/3}}\mu(k)k^{-s}$. From this we can decompose $\Lambda(n)$ for
$n\sim X/a$ as
\begin{equation*}
 \Lambda(n)=\sum_{j=1}^3\binom{3}{j}(-1)^{j-1}\sum_{k_1\dots k_{2j}=n}\mu(k_1)\cdots\mu(k_j)\log k_{2j}.
\end{equation*}
Thus, we can express $\sum\limits_{n\sim X/a}\Lambda(n)G_a(n)$ in terms of sums
\begin{equation*}
 \mathop{\sum\,\,\dots\,\,\sum}_{\substack{k_1\cdots k_{2j}\sim X/a\\ k_i\sim K_i}}\mu(k_1)\cdots\mu(k_j)
 (\log k_{2j})G_a(k_1\cdots k_{2j}),
\end{equation*}
where $1\leqslant j\leqslant3,\, K_1K_2\cdots K_{2j}\sim X/a$ and $K_1,\dots,K_j\leqslant (2X/a)^{1/3}$. By dividing the $K_j$ into two groups, we have
\begin{equation}\label{expo-fenjie}
 \Bigg|\sum_{A_1(x)\leqslant a<A_2(x)}\!\!\tau_r^s(a)\sum_{n\sim X/a}\Lambda(n)G_a(n)\Bigg|\ll_\eta X^\eta\max\Bigg|\sum_{A_1(x)\leqslant a<A_2(x)}\!\!\tau_r^s(a)
 \mathop{\sum\sum}_{\substack{kn\sim X/a \\ k\sim K}}\alpha(k)\beta(n)G_a(kn)\Bigg|,
\end{equation}
where the maximum is taken over all bilinear forms with coefficients satisfying one of
\begin{equation}\label{type-II-coeff}
   |\alpha(k)|\leqslant 1,\qquad \qquad |\beta(n)|\leqslant1,
\end{equation}
or
\begin{equation*}
   |\alpha(k)|\leqslant 1,\qquad \qquad \beta(n)=1,
\end{equation*}
or
\begin{equation*}
   |\alpha(k)|\leqslant 1,\qquad \qquad \beta(n)=\log n,
\end{equation*}
and also satisfying in all cases
\begin{equation}\label{gene-coeff-condi}
  K\leqslant X/a.
\end{equation}
We refer to the case (\ref{type-II-coeff}) as being Type II sums and to the other cases as being Type I sums and write for brevity $\Sigma_{II}$ and $\Sigma_{I}$, respectively. In the following two sections, we shall give appropriate upper bound estimates for the sums of Type II and Type I, respectively.

\section{Estimate of Type II Sums}\label{se-type-II}
We begin by putting the two variables $a$ and $k$ together, i.e. $m=ak$, and breaking up the ranges for $m$ and $h$ into intervals $(M,2M]$ and $(J,2J]$ so that $MN\asymp X$ and $\frac{1}{2}\leqslant J\leqslant H$. Then, for the Type II sums $\Sigma_{II}$, there holds
\begin{align*}
          \Sigma_{II}
 = & \,\, \mathop{\sum_{m\sim M}\sum_{n\sim N}}_{MN\sim X}
          \Bigg(\sum_{\substack{m=ak\\ A_1(x)\leqslant a<A_2(x)\\ k\sim K}}\tau_r^s(a)\alpha(k)\Bigg)\beta(n)
          \sum_{h\sim J}\Phi_h(mn)e\big(-h(mn)^\gamma\big)
                \nonumber \\
\ll & \,\, X^\eta\sum_{m\sim M}\Bigg|\sum_{\substack{n\sim N\\ mn\sim X}}
           \sum_{h\sim J}\beta(n)\Phi_h(mn)e(-h(mn)^\gamma)\Bigg|,
\end{align*}
where
\begin{equation*}
\Phi_h(mn)=\sum_{\substack{d\leqslant D\\ (d,l)=1\\ d|mn-l}}c'(d,h),\qquad |c'(d,h)|=1.
\end{equation*}
Denote by $T$ a parameter, which will be chosen later. We decompose the collection of available pairs $(n,h)$ into sets $\mathscr{S}_t$, for $1\leqslant t\leqslant T$, defined by
\begin{equation*}
  \mathscr{S}_t=\bigg\{(n,h):\,\,n\sim N,\, h\sim J, \,\frac{4JN^\gamma(t-1)}{T}<hn^\gamma\leqslant
    \frac{4JN^\gamma t}{T}\bigg\}.
\end{equation*}
Hence one has
\begin{equation*}
  \Sigma_{II}\ll X^\eta\sum_{1\leqslant t\leqslant T}\sum_{m\sim M}
  \Bigg|\mathop{\sum\sum}_{\substack{(n,h)\in\mathscr{S}_t\\ mn\sim X}}
  \beta(n)\Phi_h(mn)e(-h(mn)^\gamma)\Bigg|,
\end{equation*}
which combined with Cauchy's inequality yields
\begin{align*}
    & \,\, |\Sigma_{II}|^2 \ll X^\eta TM\sum_{1\leqslant t\leqslant T}\sum_{m\sim M}
           \Bigg|\mathop{\sum\sum}_{\substack{(n,h)\in\mathscr{S}_t\\ mn\sim X}}
           \beta(n)\Phi_h(mn)e(-h(mn)^\gamma)\Bigg|^2
                     \nonumber \\
\ll & \,\, X^\eta TM\sum_{1\leqslant t\leqslant T}\mathop{\sum\sum}_{(n_1,h_1)\in\mathscr{S}_t}
           \mathop{\sum\sum}_{(n_2,h_2)\in\mathscr{S}_t}\Bigg|\sum_{\substack{m\sim M\\ mn_1\sim X\\ mn_2\sim X}}
           \Phi_{h_1}(mn_1)\Phi_{h_2}(mn_2)e\big(\big(h_1n_1^\gamma-h_2n_2^\gamma\big)m^\gamma\big)\Bigg|
                     \nonumber \\
\ll & \,\, X^\eta TM\mathop{\sum_{n_1\sim N}\sum_{n_2\sim N}
           \sum_{h_1\sim J}\sum_{h_2\sim J}}_{|\lambda|\leqslant4JN^\gamma T^{-1}}
           \Bigg|\sum_{\substack{m\sim M\\ mn_1\sim X\\ mn_2\sim X}}\Phi_{h_1}(mn_1)\Phi_{h_2}(mn_2)
           e(\lambda m^\gamma)\Bigg|,
\end{align*}
where
\begin{equation*}
  \lambda=h_1n_1^\gamma-h_2n_2^\gamma.
\end{equation*}
Denote by $\mathcal{S}$ the innermost sum over $m$. By the definition of the quantity $\Phi_h(\cdot)$, we have
\begin{equation*}
 \mathcal{S}=\sum_{\substack{m\sim M\\ mn_1\sim X\\ mn_2\sim X}}
 \sum_{\substack{d_1\leqslant D\\ (d_1,l)=1\\ d_1|mn_1-l}}c'(d_1,h_1)
 \sum_{\substack{d_2\leqslant D\\ (d_2,l)=1\\ d_2|mn_2-l}}c'(d_2,h_2)e(\lambda m^\gamma).
\end{equation*}
If the system of the congruences
\begin{equation*}
\begin{cases}
  mn_1\equiv l\!\!\!\pmod {d_1} \\
  mn_2\equiv l\!\!\!\pmod {d_2}
\end{cases}
\end{equation*}
is not solvable, then $\mathcal{S}=0$. If the above system is solvable, then there exists some positive
integer $\mathfrak{g}=\mathfrak{g}(d_1,d_2,l,n_1,n_2)$ with $1\leqslant\mathfrak{g}\leqslant[d_1,d_2]$ such that the system is equivalent to $m\equiv\mathfrak{g}\!\pmod {[d_1,d_2]}$. In this case, we change the order of summation of $\mathcal{S}$ to derive that
\begin{equation*}
\mathcal{S}=\sum_{\substack{d_1\leqslant D\\ (d_1,l)=1}}c'(d_1,h_1)
\sum_{\substack{d_2\leqslant D\\ (d_2,l)=1}}c'(d_2,h_2)
\sum_{\substack{m\sim M\\ mn_1\sim X\\ mn_2\sim X\\ m\equiv\mathfrak{g}\!\!\!\!\!\pmod {[d_1,d_2]}}}
e\big(\lambda m^\gamma\big).
\end{equation*}
Therefore, by Lemma \ref{ex-arith-pair}, we deduce that for any exponent pair $(\kappa,\ell)$ there holds
\begin{align*}
              \mathcal{S}
 \ll & \,\, \sum_{\substack{d_1\leqslant D\\ (d_1,l)=1}}\sum_{\substack{d_2\leqslant D\\ (d_2,l)=1}}
   \Bigg|\sum_{\substack{m\sim M\\ mn_1\sim X\\ mn_2\sim X\\ m\equiv\mathfrak{g}\!\!\!\!\!\pmod {[d_1,d_2]}}}
            e(\lambda m^\gamma)\Bigg|
                 \nonumber \\
 \ll & \,\, \sum_{\substack{d_1\leqslant D\\ (d_1,l)=1}}\sum_{\substack{d_2\leqslant D\\ (d_2,l)=1}}
              \min\bigg(\frac{M}{[d_1,d_2]},\frac{M^{1-\gamma}}{|\lambda|[d_1,d_2]}+|\lambda|^{\kappa}
              [d_1,d_2]^{\kappa-\ell}M^{\kappa\gamma-\kappa+\ell}\bigg).
\end{align*}
In order to compute the contribution of the term
$|\lambda|^{\kappa}[d_1,d_2]^{\kappa-\ell}M^{\kappa\gamma-\kappa+\ell}$ to $|\Sigma_{II}|^2$, we need the following Lemma.
\begin{lemma}\label{fenduan-es}
For  $\frac{1}{2}<\alpha<1,\,J\geqslant1,\,N\geqslant1,\,\Delta>0$, let $\mathscr{N}(\Delta)$ denote the number of solutions of
the following inequality
\begin{equation*}
 \big|h_1n_1^\alpha-h_2n_2^\alpha\big|\leqslant\Delta,\qquad h_1,h_2\sim J,\quad n_1,n_2\sim N.
\end{equation*}
Then we have
\begin{equation*}
 \mathscr{N}(\Delta)\ll \Delta JN^{2-\alpha}+JN\log(JN).
\end{equation*}
\end{lemma}
\begin{proof}
 See the arguments on pp. 256--257 of Heath--Brown \cite{Heath-Brown-1983}.
\end{proof}

From the following estimate
\begin{align*}
            \sum_{d_1\leqslant D}\sum_{d_2\leqslant D}[d_1,d_2]^{\kappa-\ell}
 \ll & \,\, \sum_{d_1\leqslant D}\sum_{d_2\leqslant D}\bigg(\frac{(d_1,d_2)}{d_1d_2}\bigg)^{\ell-\kappa}
              =\sum_{1\leqslant r\leqslant D}\sum_{k_1\leqslant D/r}\sum_{k_2\leqslant D/r}
              \frac{1}{r^{\ell-\kappa}k_1^{\ell-\kappa}k_2^{\ell-\kappa}}
                 \nonumber \\
 \ll & \,\, \sum_{1\leqslant r\leqslant D}r^{\kappa-\ell}\Bigg(\sum_{k\leqslant D/r}k^{\kappa-\ell}\Bigg)^2
              \ll\sum_{1\leqslant r\leqslant D}r^{\kappa-\ell}\big(Dr^{-1}\big)^{2(\kappa-\ell+1)}\ll D^{\kappa-\ell+1},
\end{align*}
we know that the total contribution of the term  $|\lambda|^{\kappa}[d_1,d_2]^{\kappa-\ell}M^{\kappa\gamma-\kappa+\ell}$ to $|\Sigma_{II}|^2$ is
\begin{align}\label{s-upp-1}
 \ll & \,\, X^\eta TM\Bigg(\sum_{d_1\leqslant D}\sum_{d_2\leqslant D}[d_1,d_2]^{\kappa-\ell}\Bigg)
            |\lambda|^\kappa M^{\kappa\gamma-\kappa+\ell}\times
\mathop{\sum_{n_1\sim N}\sum_{n_2\sim N}\sum_{h_1\sim J}\sum_{h_2\sim J}}_{|\lambda|\leqslant4JN^\gamma T^{-1}}1
                   \nonumber \\
 \ll & \,\, X^\eta TMD^{\kappa-\ell+1}(JN^\gamma T^{-1})^\kappa
            M^{\kappa\gamma-\kappa+\ell}\cdot\mathscr{N}(4JN^\gamma T^{-1})
                   \nonumber \\
\ll & \,\,  X^\eta TM^{-\kappa+\ell+1}D^{\kappa-\ell+1}(JM^\gamma N^\gamma T^{-1})^\kappa
            \cdot\mathscr{N}(4JN^\gamma T^{-1})
                   \nonumber \\
\ll & \,\,  X^{\kappa\gamma+\eta}T^{1-\kappa}M^{-\kappa+\ell+1}D^{\kappa-\ell+1}J^\kappa
            \cdot\mathscr{N}(4JN^\gamma T^{-1}).
\end{align}
If $|\lambda|\leqslant M^{-\gamma}$, then $M[d_1,d_2]^{-1}\leqslant M^{1-\gamma}|\lambda|^{-1}[d_1,d_2]^{-1}$, and thus the total contribution of the term $M[d_1,d_2]^{-1}$ to $|\Sigma_{II}|^2$ is
\begin{equation}\label{S-B-es}
\ll X^\eta TM\cdot ML^3\times
\mathop{\sum_{n_1\sim N}\sum_{n_2\sim N}\sum_{h_1\sim J}\sum_{h_2\sim J}}_{|\lambda|\leqslant M^{-\gamma}}1
\ll X^\eta TM^2\cdot\mathscr{N}(M^{-\gamma}),
\end{equation}
where we use the elementary estimate
\begin{equation*}
\sum_{d_1\leqslant D}\sum_{d_2\leqslant D}[d_1,d_2]^{-1}\ll (\log D)^3.
\end{equation*}
If $|\lambda|> M^{-\gamma}$, then $M[d_1,d_2]^{-1}> M^{1-\gamma}|\lambda|^{-1}[d_1,d_2]^{-1}$. It follows from
the splitting argument that the  total contribution of the term $M^{1-\gamma}|\lambda|^{-1}[d_1,d_2]^{-1}$ to $|\Sigma_{II}|^2$ is
\begin{align}\label{L-B-es}
\ll & \,\, X^\eta TM\cdot M^{1-\gamma}L^3\times\max_{M^{-\gamma}\leqslant\Delta\leqslant4JN^\gamma T^{-1}}
           \mathop{\sum_{n_1\sim N}\sum_{n_2\sim N}\sum_{h_1\sim J}
           \sum_{h_2\sim J}}_{\Delta<|\lambda|\leqslant2\Delta}\frac{1}{|\lambda|}
                 \nonumber \\
\ll & \,\, X^\eta TM^{2-\gamma}\times\max_{M^{-\gamma}\leqslant\Delta\leqslant4JN^\gamma T^{-1}}
           \Delta^{-1}\cdot\mathscr{N}(2\Delta),
\end{align}
which covers the upper bound estimate (\ref{S-B-es}). From Lemma \ref{fenduan-es}, we know that
\begin{equation*}
\mathscr{N}(\Delta)\ll \Delta JN^{2-\gamma}+JNL,
\end{equation*}
which combined with (\ref{s-upp-1}) and (\ref{L-B-es}) yields
\begin{align}\label{Sigma-2-fi-1}
           \big|\Sigma_{II}\big|^2
\ll & \,\, X^{\kappa\gamma+\eta}T^{1-\kappa}M^{-\kappa+\ell+1}D^{\kappa-\ell+1}J^\kappa
           \cdot\mathscr{N}(4JN^\gamma T^{-1})
              \nonumber \\
& \,\, \qquad +X^\eta TM^{2-\gamma}\times\max_{M^{-\gamma}\leqslant\Delta\leqslant4JN^\gamma T^{-1}}
       (JN^{2-\gamma}+JN\Delta^{-1})
              \nonumber \\
\ll & \,\, X^\eta\big(X^{\kappa\gamma+2}T^{-\kappa}M^{-\kappa+\ell-1}D^{\kappa-\ell+1}J^{\kappa+2}
           +X^{\kappa\gamma+1}T^{1-\kappa}M^{-\kappa+\ell}D^{\kappa-\ell+1}J^{\kappa+1}
              \nonumber \\
    & \,\, \qquad +TJX^{2-\gamma}+TMXJ\big).
\end{align}
We choose $T$ such that the first term and the fourth term in the above estimate are equal. Accordingly, we take
\begin{equation}\label{T-chosen}
T=\big[X^{\frac{\kappa\gamma+1}{\kappa+1}}M^{\frac{-\kappa+\ell-2}{\kappa+1}}
D^{\frac{\kappa-\ell+1}{\kappa+1}}J\big]+1.
\end{equation}
Putting (\ref{T-chosen}) into (\ref{Sigma-2-fi-1}), we obtain
\begin{align*}
           \big|\Sigma_{II}\big|^2
\ll & \,\, X^\eta\Big(X^{\frac{\kappa(\gamma+1)+2}{\kappa+1}}M^{\frac{\ell-1}{\kappa+1}}
           D^{\frac{\kappa-\ell+1}{\kappa+1}}J^2
           +X^{\frac{2(\kappa\gamma+1)}{\kappa+1}}M^{\frac{2(\ell-1)}{\kappa+1}}
           D^{\frac{2(\kappa-\ell+1)}{\kappa+1}}J^2
                          \nonumber \\
    & \,\, \qquad+X^{\frac{2\kappa+3-\gamma}{\kappa+1}}M^{\frac{-\kappa+\ell-2}{\kappa+1}}
           D^{\frac{\kappa-\ell+1}{\kappa+1}}J^2
           +X^{\kappa\gamma+1}M^{-\kappa+\ell}D^{\kappa-\ell+1}J^{\kappa+1}
           +JX^{2-\gamma}+MXJ\Big),
\end{align*}
which combined with $J\ll H=X^{1-\gamma+\eta}$ and $D\leqslant X^{\xi+\eta/2}$ for $\xi\leqslant(1-\eta)/2$ yields
\begin{align*}
           \big|\Sigma_{II}\big|^2
\ll & \,\, X^\eta\Big(X^{\frac{3\kappa+4-(\kappa+2)\gamma}{\kappa+1}+\frac{\kappa-\ell+1}{\kappa+1}\xi}
           M^{\frac{\ell-1}{\kappa+1}}
           +X^{\frac{2(\kappa+2-\gamma)}{\kappa+1}+\frac{2(\kappa-\ell+1)}{\kappa+1}\xi}
           M^{\frac{2(\ell-1)}{\kappa+1}}
              \nonumber \\
    & \,\, \qquad+X^{\frac{4\kappa+5-(2\kappa+3)\gamma}{\kappa+1}+\frac{\kappa-\ell+1}{\kappa+1}\xi}
           M^{\frac{-\kappa+\ell-2}{\kappa+1}}
           +X^{\kappa+2-\gamma+(\kappa-\ell+1)\xi}M^{-\kappa+\ell}
           +X^{3-2\gamma}+MX^{2-\gamma}\Big).
\end{align*}
According to above arguments, we deduce the following lemma.
\begin{lemma}\label{Type-II-es}
Assume that $(\kappa,\ell)$ is an exponent pair. Suppose that $\frac{1}{2}<\gamma<1$ and $0<\xi\leqslant(1-\eta)/2$ subject to the conditions
\begin{equation}\label{Type-II-condi-1}
 \gamma>\frac{\kappa}{\kappa-\ell+1}+\xi+\eta
\end{equation}
and
\begin{equation}\label{Type-II-condi-2}
 \gamma>\frac{\kappa+2}{\kappa-\ell+3}+\frac{\kappa-\ell+1}{\kappa-\ell+3}\xi+\eta.
\end{equation}
If there holds
\begin{equation*}
 X^{\frac{(\kappa+2)(1-\gamma)+(\kappa-\ell+1)\xi}{1-\ell}+\eta}\ll M\ll X^{\gamma-\eta},
\end{equation*}
then we have
\begin{equation*}
 \Sigma_{II}\ll X^{1-\eta}.
\end{equation*}
\end{lemma}

\section{Estimate of Type I Sums}\label{se-type-I}
As in section \ref{se-type-II}, we also put $m=ak$ and breaking up the range for $m$ into intervals $(M,2M]$ such that $MN\asymp X$. According to the definition of the quantity $\Xi_h(\cdot)$, one derives that
\begin{equation*}
\Sigma_I=\mathop{\sum_{m\sim M}\sum_{n\sim N}}_{MN\sim X}\varpi(m)\beta(n)\sum_{0<|h|\leqslant H}e(-h(mn)^\gamma)
\sum_{\substack{d\leqslant D\\ (d,l)=1\\ d|mn-l}}c''(d,h),
\end{equation*}
where
\begin{equation*}
\varpi(m)=\sum_{\substack{m=ak\\ A_1(x)\leqslant a<A_2(x)\\ k\sim K}}\tau_r^s(a)\alpha(k)\ll X^\eta.
\end{equation*}
By changing the order of summation we obtain
\begin{equation}\label{Type-1-fenjie}
 \Sigma_I\ll L\sum_{0<h\leqslant H}\mathcal{H}_h,
\end{equation}
where
\begin{equation*}
\mathcal{H}_h=\sum_{\substack{d\leqslant D\\ (d,l)=1}}c''(d,h)\sum_{m\sim M}\varpi(m)
\sum_{\substack{n\sim N\\ mn\sim X\\ mn\equiv l \!\!\!\!\!\pmod d}}e(h(mn)^\gamma).
\end{equation*}
By Lemma \ref{ex-arith-pair}, we deduce that for any exponent pair $(\kappa,\ell)$ there holds
\begin{align}\label{H_h-sin}
           \mathcal{H}_h
\ll & \,\, X^\eta\sum_{\substack{d\leqslant D\\ (d,l)=1}}\sum_{m\sim M}
           \Bigg|\sum_{\substack{n\sim N\\ mn\sim X\\ mn\equiv l\!\!\!\!\!\pmod d}}e(h(mn)^\gamma)\Bigg|
                   \nonumber \\
\ll & \,\, X^\eta\sum_{\substack{d\leqslant D\\ (d,l)=1}}\sum_{m\sim M}\Big(h^{-1}X^{1-\gamma}M^{-1}d^{-1}+
           h^{\kappa}M^{\kappa-\ell}d^{\kappa-\ell}X^{\kappa\gamma-\kappa+\ell}\Big)
                    \nonumber \\
\ll & \,\, h^{-1}X^{1-\gamma+\eta}+h^{\kappa}X^{\kappa\gamma-\kappa+\ell+\eta}M^{\kappa-\ell+1}D^{\kappa-\ell+1}.
\end{align}
From (\ref{Type-1-fenjie}) and (\ref{H_h-sin}), one has
\begin{align*}
           \Sigma_I
\ll & \,\, X^{1-\gamma+\eta}+H^{\kappa+1}X^{\kappa\gamma-\kappa+\ell}M^{\kappa-\ell+1}D^{\kappa-\ell+1}
                \nonumber \\
\ll & \,\, X^{1-\gamma+\eta}+M^{\kappa-\ell+1}X^{1+\ell-\gamma+(\kappa-\ell+1)\xi+\eta}.
\end{align*}
According to above estimate, we obtain the following lemma.
\begin{lemma}\label{Type-I-es}
Assume that $(\kappa,\ell)$ is an exponent pair. Suppose that $M$ satisfies the condition
\begin{equation*}
 M\ll X^{\frac{\gamma-\ell}{\kappa-\ell+1}-\xi-\eta}.
\end{equation*}
Then we have
\begin{equation*}
 \Sigma_{I}\ll X^{1-\eta}.
\end{equation*}
\end{lemma}

\section{Proof of Theorem \ref{Theorem-1}}
As the illustration in section \ref{se-type-II} and section \ref{se-type-I}, during the process of the treatment of exponential sum estimate, we regard $m=ak$ as one variable and $n$ as another. Hence one can deduce from (\ref{expo-fenjie}) that
\begin{equation}\label{expo-decom}
 \Bigg|\sum_{A_1(x)\leqslant a<A_2(x)}\!\!\tau_r^s(a)\sum_{n\sim X/a}\Lambda(n)G_a(n)\Bigg|\ll_\eta X^\eta\max\Bigg|\mathop{\sum_{m\sim M}\sum_{n\sim N}}_{MN\sim X}\varpi(m)\beta(n)\mathcal{F}(mn)\Bigg|,
\end{equation}
where the maximum is taken over all bilinear forms with coefficients satisfying one of
\begin{equation}\label{type-II-coeff-N}
   |\varpi(m)|\leqslant 1,\qquad \qquad |\beta(n)|\leqslant1,
\end{equation}
or
\begin{equation*}
   |\varpi(m)|\leqslant 1,\qquad \qquad \beta(n)=1,
\end{equation*}
or
\begin{equation*}
   |\varpi(m)|\leqslant 1,\qquad \quad \beta(n)=\log n,
\end{equation*}
and also satisfying in all cases
\begin{equation}\label{gene-coeff-condi-N}
  M\leqslant X.
\end{equation}
In the argument between (\ref{suffi-genera})--(\ref{type-II-coeff}), by dividing the $K_j$ into two groups in a judicious fashion, with making a new variable $m$ by bonding $a$ and corresponding $k_j$ in one group, we are able to reduce the range of $M$ from (\ref{gene-coeff-condi-N}).
\begin{lemma}\label{exponen-fenjie}
If we have real numbers $0<\mathfrak{a}<1,\,0<\mathfrak{b}<\mathfrak{c}<1$  satisfying
\begin{equation*}
\mathfrak{b}<\frac{2}{3},\qquad 1-\mathfrak{c}<\mathfrak{c}-\mathfrak{b}, \qquad 1-\mathfrak{a}<\frac{\mathfrak{c}}{2},
\end{equation*}
then (\ref{expo-decom}) still holds when (\ref{gene-coeff-condi-N}) is replaced by the conditions
\begin{equation*}
M\leqslant X^\mathfrak{a} \qquad \textrm{for Type I sums},
\end{equation*}
and
\begin{equation*}
X^\mathfrak{b}\leqslant M\leqslant X^\mathfrak{c} \qquad \textrm{for Type II sums}.
\end{equation*}
\end{lemma}
\begin{proof}
 See Proposition 1 of Balog and Friedlander \cite{Balog-Friedlander-1992}.
\end{proof}
In order to prove (\ref{suffi-2}), we need to take appropriate exponent pair $(\kappa,\ell)$ in Lemma \ref{Type-II-es} and Lemma \ref{Type-I-es}, respectively. It is easy to see that, for $n\geqslant1$, there holds
\begin{equation*}
A^n\bigg(\frac{1}{2},\frac{1}{2}\bigg)=\bigg(\frac{1}{2^{n+2}-2},1-\frac{n+1}{2^{n+2}-2}\bigg).
\end{equation*}
In Lemma \ref{Type-II-es}, we take $(\kappa,\ell)=A^3(\frac{1}{2},\frac{1}{2})=(\frac{1}{30},\frac{13}{15})$.
Then the conditions (\ref{Type-II-condi-1}) and (\ref{Type-II-condi-2}) make
\begin{equation}\label{spe-Type-II-condi-ex-1}
\gamma>\frac{1}{5}+\xi+\eta
\end{equation}
and
\begin{equation}\label{spe-Type-II-condi-ex-2}
 \gamma>\frac{61}{65}+\frac{1}{13}\xi+\eta.
\end{equation}
Hence $\Sigma_{II}\ll X^{1-\eta}$ provided that
\begin{equation*}
X^{\frac{61(1-\gamma)+5\xi}{4}+\eta}\ll M\ll X^{\gamma-\eta}.
\end{equation*}
In Lemma \ref{Type-I-es}, we take
\begin{equation*}
(\kappa,\ell)=A^{36}\bigg(\frac{1}{2},\frac{1}{2}\bigg)=\bigg(\frac{1}{2^{38}-2},1-\frac{37}{2^{38}-2}\bigg).
\end{equation*}
Then $\Sigma_{I}\ll X^{1-\eta}$ provided that
\begin{equation*}
M\ll X^{\frac{2^{37}-1}{19}\gamma-\frac{2^{38}-39}{38}-\xi-\eta}.
\end{equation*}
From Lemma \ref{exponen-fenjie}, we take
\begin{align*}
\mathfrak{a} = & \,\, \frac{2^{37}-1}{19}\gamma-\frac{2^{38}-39}{38}-\xi-\eta, \\
\mathfrak{b} = & \,\, \frac{61(1-\gamma)+5\xi}{4}+\eta,  \\
\mathfrak{c} = & \,\, \gamma-\eta.
\end{align*}
Then one can easily check that the conditions (\ref{suffi-condi-1}), (\ref{spe-Type-II-condi-ex-1}) and
(\ref{spe-Type-II-condi-ex-2}), as well as the inequalities in Lemma \ref{exponen-fenjie}, hold. Therefore, we
show that (\ref{suffi-genera}) holds, which is sufficient to complete the proof of (\ref{suffi-2}).

For (\ref{suffi-3}), we can follow the process from (\ref{suffi-4}) to (\ref{suffi-condi-1}). Then it suffices to
show that
\begin{align}\label{suffi-S-3-condi}
 S^*:= & \,\, \sum_{\substack{d\leqslant D\\ (d,l)=1}}\frac{1}{\varphi(d)}\sum_{0<h\leqslant H}\frac{1}{h}
            \sum_{\substack{A_1(x)\leqslant a<A_2(x)\\ (a,d)=1}}\tau_r^s(a)
                  \nonumber \\
 & \,\, \qquad \times\Bigg|\sum_{an\sim X}
        \Lambda(n)\Big(e\big(-h(an)^\gamma\big)-e\big(-h(an+1)^\gamma\big)\Big)\Bigg|\ll x^\gamma L^{-A}.
\end{align}
Afterwards, for the innermost sum on the right--hand side of (\ref{suffi-S-3-condi}), one can follow the routine
process as is shown in (\ref{inner-sum-upper}) to see that
\begin{align*}
S^* \ll & \,\, X^{\gamma-1}\times\sum_{\substack{d\leqslant D\\ (d,l)=1}}\frac{1}{\varphi(d)}
               \sum_{0<h\leqslant H}
               \sum_{\substack{A_1(x)\leqslant a<A_2(x)\\ (a,d)=1}}\tau_r^s(a)
               \Bigg|\sum_{an\sim X}\Lambda(n)e(-h(an)^\gamma)\Bigg|
                    \nonumber \\
      = & \,\, X^{\gamma-1}\times\sum_{\substack{d\leqslant D\\ (d,l)=1}}\frac{1}{\varphi(d)}\sum_{0<h\leqslant H}
               \sum_{\substack{A_1(x)\leqslant a<A_2(x)\\ (a,d)=1}}\tau_r^s(a)\cdot c^*(d,h,a)
               \sum_{an\sim X}\Lambda(n)e(-h(an)^\gamma)
                    \nonumber \\
    \ll & \,\, X^{\gamma-1}\times\sum_{A_1(x)\leqslant a<A_2(x)}\tau_r^s(a)
               \sum_{an\sim X}\Lambda(n)\sum_{0<h\leqslant H}e(-h(an)^\gamma)
               \sum_{\substack{d\leqslant D\\ (d,l)=1}}\frac{c^*(d,h,a)}{\varphi(d)}
                    \nonumber \\
      = & \,\, X^{\gamma-1}\times\sum_{A_1(x)\leqslant a<A_2(x)}\tau_r^s(a)\sum_{an\sim X}\Lambda(n)G_a^*(n),
\end{align*}
where
\begin{equation*}
  G_a^*(n)=\sum_{0<h\leqslant H}\Xi_{h,a}^*(n)e(-h(an)^\gamma)
\end{equation*}
and
\begin{equation*}
  \Xi_{h,a}^*(n)=\sum_{\substack{d\leqslant D\\ (d,l)=1}}\frac{c^*(d,h,a)}{\varphi(d)}, \qquad |c^*(d,h,a)|=1.
\end{equation*}
Hence, in order to show (\ref{suffi-3}), we only need to prove
\begin{equation*}
  \Bigg|\sum_{A_1(x)\leqslant a<A_2(x)}\tau_r^s(a)\sum_{an\sim X}\Lambda(n)G_a^*(n)\Bigg|\ll XL^{-A},
\end{equation*}
which can be treated almost exactly the same as the arguments in dealing with (\ref{suffi-2}). Consequently, we omit the details herein. This completes the proof of Theorem \ref{Theorem-1}.

\section{Proof of Theorem \ref{Theorem-2}}

In this section, we shall prove Theorem \ref{Theorem-2} according to the result of Theorem \ref{Theorem-1},
Corollary \ref{Coro-1}, (\ref{Thm-1-eq-shifted}), the `weighted sieve' and the `switching principle' constructed by Chen \cite{Chen-1973}. First, we list some lemmas which is necessary for proving Theorem \ref{Theorem-2}.

\subsection{Some Preliminary Lemmas}

Let $\mathscr{A}$ denote a finite set of integers, $\mathscr{P}$ denote an infinite set of primes,
$\overline{\mathscr{P}}$ denote the set of primes which do not belong to $\mathscr{P}$. For a given $z\geqslant2$, we set
\begin{equation*}
P(z)=\prod_{\substack{p<z\\ p\in\mathscr{P}}}p.
\end{equation*}
Define the sifting function as
\begin{equation*}
S(\mathscr{A},\mathscr{P},z)=\big|\big\{a\in\mathscr{A}:(a,P(z))=1\big\}\big|.
\end{equation*}
For $d|P(z)$, define $\mathscr{A}_d=\{a\in\mathscr{A}:a\equiv0\!\pmod d\}$.
Moreover, we assume that $|\mathscr{A}_d|$ may be written in the form
\begin{equation}\label{sieve-condi-1}
|\mathscr{A}_d|=\frac{\omega(d)}{d}X+r_d,\qquad \mu(d)\not=0,\qquad (d,\overline{\mathscr{P}})=1,
\end{equation}
where $\omega(d)$ is a multiplicative function such that $0\leqslant\omega(p)<p$, $X$ is a positive number
independent of $d$, and $r_d$ is an error term which is to be small on average so that $X$ approximates
to the cardinality of $\mathscr{A}$. Also, we assume that the function $\omega(p)$ is constant on average
over $p$ in $\mathscr{P}$, which means that
\begin{equation}\label{sieve-condi-2}
\sum_{\substack{z_1\leqslant p<z_2\\ p\in\mathscr{P}}}\bigg(1-\frac{\omega(p)}{p}\bigg)^{-1}
\leqslant\frac{\log z_2}{\log z_1}\bigg(1+\frac{\mathcal{K}}{\log z_1}\bigg)
\end{equation}
holds for all $z_2>z_1\geqslant2$, where $\mathcal{K}$ is a constant satisfying $\mathcal{K}\geqslant1$. For details of (\ref{sieve-condi-1}) and (\ref{sieve-condi-2}), one can see the arguments (4.12)--(4.15) on page 28 of Halberstam and Richert \cite{Halberstam-Richert-book}, and the arguments on page 205 of Iwaniec \cite{Iwaniec-1981}.
\begin{lemma}
Suppose that the conditions (\ref{sieve-condi-1}) and (\ref{sieve-condi-2}) hold. Then we have
\begin{align}
 S(\mathscr{A},\mathscr{P},z) \geqslant & \,\, XV(z)\big(f(s)+O\big(\log^{-1/3}D\big)\big)-R_D, \label{lower-sieve}
               \\
 S(\mathscr{A},\mathscr{P},z) \leqslant & \,\, XV(z)\big(F(s)+O\big(\log^{-1/3}D\big)\big)+R_D, \label{upper-sieve}
\end{align}
where
\begin{equation*}
R_D=\sum_{\substack{d<D\\ d|P(z)}}|r_d|,\qquad s=\frac{\log D}{\log z},
\end{equation*}
\begin{equation}\label{V(z)-def}
V(z)=\mathcal{C}(\omega)\frac{e^{-C_0}}{\log z}\bigg(1+O\bigg(\frac{1}{\log z}\bigg)\bigg),
\end{equation}
\begin{equation}\label{C(omega)-def}
\mathcal{C}(\omega)=\prod_p\bigg(1-\frac{\omega(p)}{p}\bigg)\bigg(1-\frac{1}{p}\bigg)^{-1},
\end{equation}
where $C_0$ denotes the Euler's constant, $f(s)$ and $F(s)$ denote the classical functions in the linear sieve
theory, which are determined by the following differential--difference equation
\begin{equation}
\begin{cases}\label{diff-eq}
F(s)=\displaystyle\frac{2e^{C_0}}{s},\quad f(s)=0, \quad 0<s\leqslant2,\\
\displaystyle\frac{\mathrm{d}}{\mathrm{d}s}(sF(s))=f(s-1),\quad \displaystyle\frac{\mathrm{d}}{\mathrm{d}s}(sf(s))=F(s-1),\quad s\geqslant2.
\end{cases}
\end{equation}
\end{lemma}
\begin{proof}
For (\ref{lower-sieve}) and (\ref{upper-sieve}), one can refer to (6), (7), (8) on page 209 of Iwaniec \cite{Iwaniec-1981}, while (\ref{diff-eq}) can be referred to as a special case with $\varkappa=1,\beta=2$ in
(9) of Iwaniec \cite{Iwaniec-1981}. Moreover, for (\ref{V(z)-def}) and (\ref{C(omega)-def}) one can see (2.4)
and (2.5) of Chapter $5$ in Halberstam and Richert \cite{Halberstam-Richert-book}.
\end{proof}

\begin{lemma}
Let $F(s)$ and $f(s)$ be defined as in (\ref{diff-eq}). Then we have
\begin{align*}
F(s) = & \,\, \frac{2e^{C_0}}{s},\qquad 0<s\leqslant3;
                \nonumber \\
F(s) = & \,\, \frac{2e^{C_0}}{s}\bigg(1+\int_{2}^{s-1}\frac{\log(t-1)}{t}\mathrm{d}t\bigg), \qquad
              3\leqslant s\leqslant5;
                \nonumber \\
f(s) = & \,\, \frac{2e^{C_0}\log(s-1)}{s},\qquad 2\leqslant s\leqslant4;
                \nonumber \\
f(s) = & \,\, \frac{2e^{C_0}}{s}\bigg(\log(s-1)+\int_3^{s-1}\frac{\mathrm{d}t_1}{t_1}\int_{2}^{t_1-1}
              \frac{\log(t_2-1)}{t_2}\mathrm{d}t_2\bigg),
              \qquad 4\leqslant s\leqslant6;
\end{align*}
where $C_0$ denotes Euler's constant.
\end{lemma}
\begin{proof}
See (7.51)--(7.54) on page 127 of Chapter VII in Pan and Pan \cite{Pan-Pan-book}.
\end{proof}

\subsection{Weighted Sieve Method}
Let $x$ be a sufficiently large number and set
\begin{equation*}
\mathscr{A}=\big\{a:a\leqslant x,\,a=p+2,\, a=[k^{1/\gamma}]\big\},\qquad \mathscr{P}=\big\{p:p>2\big\}.
\end{equation*}
Denote by $\mathds{1}_{\mathcal{A}}(n)$ the characteristic function of set $\mathcal{A}$, i.e.,
\begin{equation*}
\mathds{1}_{\mathcal{A}}(n)=
\begin{cases}
1, & \textrm{if}\,\, n\in\mathcal{A},  \\
0, & \textrm{if}\,\, n\not\in\mathcal{A}.
\end{cases}
\end{equation*}
Then one has
\begin{equation*}
\big|\big\{a:a\in\mathscr{A},\,\,a=\mathcal{P}_2\big\}\big|
\geqslant\sum_{\substack{a\in\mathscr{A}\\ (a,P(x^{1/10.92}))=1}}\mathds{1}_{\mathcal{P}_2}(a).
\end{equation*}
Trivially, we have
\begin{equation*}
\sum_{\substack{a\in\mathscr{A}\\ (a,P(x^{1/10.92}))=1\\ \mu(a)=0}}\mathds{1}_{\mathcal{P}_2}(a)
\ll \sum_{x^{1/10.92}\leqslant p\leqslant x^{1/2}}\frac{x}{p^2}\ll x^{1-\frac{1}{10.92}}.
\end{equation*}
Hence one gets
\begin{equation}\label{lower-tran-1}
\big|\big\{a:a\in\mathscr{A},\,\,a=\mathcal{P}_2\big\}\big|\geqslant
\sum_{\substack{a\in\mathscr{A}\\ (a,P(x^{1/10.92}))=1\\ (a,2)=1}}\mu^2(a)\mathds{1}_{\mathcal{P}_2}(a)
+O(x^{1-\frac{1}{10.92}}).
\end{equation}
Now, we claim that, for given integer $a$ with $a\leqslant x$ which subjects to $(a,P(x^{1/10.92}))=1,(a,2)=1$ and $\mu(a)\not=0$, there holds
\begin{equation}\label{weight-ineq}
\mathds{1}_{\mathcal{P}_2}(a)\geqslant1-\frac{1}{2}\varrho_1(a)-\frac{1}{2}\varrho_2(a)-\varrho_3(a),
\end{equation}
where
\begin{align*}
         \varrho_1(a)
= & \,\, \sum_{\substack{x^{1/10.92}\leqslant p<x^{1/3.29}\\ p|a}}1;
               \nonumber \\
         \varrho_2(a)
= & \,\, \begin{cases}
1, & \textrm{if}\,\, a=p_1p_2p_3,\,\,x^{1/10.92}\leqslant p_1<x^{1/3.29}\leqslant p_2<p_3,\,\,(a,2)=1, \\
0, & \textrm{otherwise};
\end{cases}
               \nonumber \\
         \varrho_3(a)
= & \,\, \begin{cases}
1, & \textrm{if}\,\, a=p_1p_2p_3,\,\,x^{1/3.29}\leqslant p_1<p_2<p_3,\,\,(a,2)=1, \\
0, & \textrm{otherwise}.
\end{cases}
\end{align*}
Actually, by noting the fact that $\varrho_1(a)\geqslant0,\varrho_2(a)\geqslant0$ and $\varrho_3(a)\geqslant0$, if
$\Omega(a)\leqslant2$, then obviously we have $\mathds{1}_{\mathcal{P}_2}(a)=1\geqslant1-\frac{1}{2}\varrho_1(a)-\frac{1}{2}\varrho_2(a)-\varrho_3(a)$. Now,
we always assume that $\Omega(a)\geqslant3$, then $\mathds{1}_{\mathcal{P}_2}(a)=0$. If $\varrho_1(a)\geqslant2$, then $\mathds{1}_{\mathcal{P}_2}(a)=0\geqslant1-\frac{1}{2}\varrho_1(a)\geqslant1-\frac{1}{2}\varrho_1(a)
-\frac{1}{2}\varrho_2(a)-\varrho_3(a)$. If $\varrho_1(a)=1$, then $a$ must not be represented as the form $a=p_1p_2p_3p_4$ with $x^{1/10.92}\leqslant p_1<x^{1/3.29}\leqslant p_2<p_3<p_4$. Otherwise, one deduces that
$a=p_1p_2p_3p_4\geqslant x^{1/10.92+3/3.29}>x$, which is a contradiction. Hence $a$ must be written as
$a=p_1p_2p_3$ with $x^{1/10.92}\leqslant p_1<x^{1/3.29}\leqslant p_2<p_3$ so that $\varrho_2(a)=1$ and $\varrho_3(a)=0$. In this case, there holds $\mathds{1}_{\mathcal{P}_2}(a)=0=1-\frac{1}{2}\varrho_1(a)-\frac{1}{2}\varrho_2(a)-\varrho_3(a)$.
If $\varrho_1(a)=0$, then $\varrho_2(a)=0$, and thus every prime factor of $a$ is not less than $x^{1/3.29}$ which
combined $\Omega(a)\geqslant3$ makes $a=p_1p_2p_3$ with $x^{1/3.29}\leqslant p_1<p_2<p_3$, i.e., $\varrho_3(a)=1$.
At this time, we also have $\mathds{1}_{\mathcal{P}_2}(a)=0=1-\frac{1}{2}\varrho_1(a)-\frac{1}{2}\varrho_2(a)-\varrho_3(a)$.
Above all, (\ref{weight-ineq}) holds under the conditions given as above.

On the other hand, one has
\begin{equation*}
\sum_{\substack{a\in\mathscr{A}\\ (a,P(x^{1/10.92}))=1\\ \mu(a)=0\\ (a,2)=1}}
\bigg(1-\frac{1}{2}\varrho_1(a)-\frac{1}{2}\varrho_2(a)-\varrho_3(a)\bigg)
\ll \sum_{x^{1/10.92}\leqslant p\leqslant x^{1/2}}\frac{x}{p^2}\ll x^{1-\frac{1}{10.92}},
\end{equation*}
which combined (\ref{lower-tran-1}) and (\ref{weight-ineq}) yields
\begin{align}\label{lower-tran-2}
                 \big|\big\{a:a\in\mathscr{A},\,\,a=\mathcal{P}_2\big\}\big|
\geqslant & \,\, \sum_{\substack{a\in\mathscr{A}\\ (a,P(x^{1/10.92}))=1\\ (a,2)=1}}
                 \bigg(1-\frac{1}{2}\varrho_1(a)-\frac{1}{2}\varrho_2(a)-\varrho_3(a)\bigg)
                 +O\big(x^{\frac{9.92}{10.92}}\big)
                    \nonumber \\
      = & \,\, S-\frac{1}{2}S_1-\frac{1}{2}S_2-S_3+O\big(x^{\frac{9.92}{10.92}}\big),
\end{align}
where
\begin{equation*}
S=\sum_{\substack{a\in\mathscr{A}\\ (a,P(x^{1/10.92}))=1\\ (a,2)=1}}1;\qquad\qquad
S_i=\sum_{\substack{a\in\mathscr{A}\\ (a,P(x^{1/10.92}))=1\\ (a,2)=1}}\varrho_i(a),\qquad i=1,2,3.
\end{equation*}
Now, we shall give the lower bound estimate of $S$ and the upper bound estimates of $S_i\,(i=1,2,3)$.

\subsubsection{Lower Bound Estimate for $S$}

We use (\ref{lower-sieve}) to  give the lower bound of $S$. Hence we take
\begin{equation*}
X=\pi_\gamma(x),\qquad D=x^\xi,\qquad
\omega(d)=
 \begin{cases}
   \displaystyle\frac{d}{\varphi(d)}, & \textrm{if $(d,2)=1$ and $\mu(d)\not=0$},\\
   \,\,\,\,\,0, & \textrm{otherwise}.
 \end{cases}
\end{equation*}
Then we have
\begin{equation*}
 \mathcal{C}(\omega)=\prod_p\bigg(1-\frac{\omega(p)}{p}\bigg)\bigg(1-\frac{1}{p}\bigg)^{-1}
 =2\prod_{p>2}\bigg(1-\frac{1}{(p-1)^2}\bigg)=:2\mathfrak{S},
\end{equation*}
say. It follows from (\ref{Thm-1-eq-shifted}) with (\ref{g-spe}) that
\begin{align*}
R_D = & \,\, \sum_{\substack{d<x^\xi\\ d|P(x^{1/10.92})}}
             \Bigg|\sum_{\substack{p\leqslant x-2\\ p\equiv-2\!\!\!\!\!\pmod d\\ p+2=[k^{1/\gamma}]}}1
             -\frac{1}{\varphi(d)}\sum_{\substack{p\leqslant x-2\\ p+2=[k^{1/\gamma}]}}1\Bigg|
                   \nonumber \\
\leqslant & \,\, \sum_{\substack{d\leqslant x^\xi\\ (d,2)=1}}
             \Bigg|\sum_{\substack{p\leqslant x\\ p\equiv-2\!\!\!\!\!\pmod d\\ p+2=[k^{1/\gamma}]}}1
             -\frac{1}{\varphi(d)}\sum_{\substack{p\leqslant x\\ p+2=[k^{1/\gamma}]}}1\Bigg|
             \ll\frac{x^\gamma}{(\log x)^A}.
\end{align*}
Then (\ref{lower-sieve}) gives
\begin{align}\label{S-lower-final}
S \geqslant & \,\,\pi_\gamma(x)V\big(x^{1/10.92}\big)\big(f(10.92\xi)+O\big(\log^{-1/3}x\big)\big)-R_D
                      \nonumber \\
\geqslant & \frac{4\mathfrak{S}}{\xi}\bigg(\log(10.92\xi-1)
            +\int_3^{10.92\xi-1}\frac{\mathrm{d}t_1}{t_1}\int_2^{t_1-1}\frac{\log(t_2-1)}{t_2}
            \mathrm{d}t_2\bigg)\frac{x^\gamma}{\log^2x}(1+o(1)).
\end{align}

\subsubsection{Upper Bound Estimate for $S_1$}

By the definition of $S_1$, we have
\begin{align}\label{S_1-upper-1}
S_1 = & \,\, \sum_{\substack{a\in\mathscr{A}\\ (a,P(x^{1/10.92}))=1\\ (a,2)=1}}
             \sum_{\substack{x^{1/10.92}\leqslant p<x^{1/3.29}\\ p|a}}1
             =\sum_{x^{1/10.92}\leqslant p<x^{1/3.29}}\sum_{\substack{a\in\mathscr{A}\\ (a,P(x^{1/10.92}))=1\\
             p|a,\,\, (a,2)=1}}1
                   \nonumber \\
\leqslant & \,\, \sum_{x^{1/10.92}\leqslant p<x^{1/3.29}}
             \sum_{\substack{a\in\mathscr{A}\\ (a,P(x^{1/10.92}))=1\\ p|a}}1
             =:\sum_{x^{1/10.92}\leqslant p<x^{1/3.29}}S(\mathscr{A}_p,\mathscr{P},x^{1/10.92}),
\end{align}
say. First we use (\ref{upper-sieve}) to  give the upper bound of $S(\mathscr{A}_p,\mathscr{P},x^{1/10.92})$. Hence for $x^{1/10.92}\leqslant p<x^{1/3.29}$, one has
\begin{align}\label{S_1-inner-1}
         & \,\,        S(\mathscr{A}_p,\mathscr{P},x^{1/10.92})
\leqslant \frac{1}{\varphi(p)}\pi_\gamma(x)V\big(x^{1/10.92}\big)
                 \Bigg(F\bigg(\frac{\log(x^\xi/p)}{\log x^{1/10.92}}\bigg)+O\big(\log^{-1/3}x\big)\Bigg)+R_D(p)
                 \nonumber \\
\leqslant & \,\, 21.84\mathfrak{S}e^{-C_0}\cdot\frac{1}{p-1}\cdot F\bigg(10.92\xi-\frac{10.92\log p}{\log x}\bigg)
                 \frac{x^\gamma}{\log^2x}(1+o(1))+R_D(p),
\end{align}
where $C_0$ is Euler's constant and
\begin{equation*}
R_D(p)=\sum_{\substack{d<x^\xi/p\\ d|P(x^{1/10.92})}}\big|r_{dp}\big|.
\end{equation*}
It follows from (\ref{Thm-1-eq-shifted}) with (\ref{g-spe}) that
\begin{align}\label{S_1-error-total}
          \sum_{x^{1/10.92}\leqslant p<x^{1/3.29}}R_D(p)
 = & \,\, \sum_{x^{1/10.92}\leqslant p<x^{1/3.29}}\sum_{\substack{d<x^\xi/p\\ d|P(x^{1/10.92})}}
 \Bigg|\sum_{\substack{q\leqslant x-2\\ q\equiv-2\!\!\!\!\!\pmod{dp}\\ q+2=[k^{1/\gamma}]}}1
  -\frac{1}{\varphi(dp)}\sum_{\substack{q\leqslant x-2\\ q+2=[k^{1/\gamma}]}}1\Bigg|
                \nonumber \\
 \leqslant & \,\, \sum_{\substack{d'\leqslant x^\xi\\ (d',2)=1}}
             \Bigg|\sum_{\substack{q\leqslant x-2\\ q\equiv-2\!\!\!\!\!\pmod {d'}\\ q+2=[k^{1/\gamma}]}}1
             -\frac{1}{\varphi(d')}\sum_{\substack{q\leqslant x-2\\ q+2=[k^{1/\gamma}]}}1\Bigg|
             \ll\frac{x^\gamma}{(\log x)^A}.
\end{align}
It follows from (\ref{S_1-upper-1})--(\ref{S_1-error-total}), prime number theorem, and partial summation that
\begin{align}\label{S_1-upper-final}
S_1\leqslant & \,\, 21.84\mathfrak{S}e^{-C_0}\frac{x^\gamma}{\log^2 x}(1+o(1))\times
                    \sum_{x^{1/10.92}\leqslant p<x^{1/3.29}}\frac{1}{p}\cdot
                    F\bigg(10.92\xi-\frac{10.92\log p}{\log x}\bigg)
                    \nonumber \\
 = & \,\, 4\mathfrak{S}(1+o(1))\frac{x^\gamma}{\log^2x}\times\Bigg(\int_{\xi-\frac{3}{10.92}}^{\frac{1}{3.29}}
          \frac{\mathrm{d}\alpha}{\alpha(\xi-\alpha)}
                    \nonumber \\
  & \,\,\qquad +\int_{\frac{1}{10.92}}^{\xi-\frac{3}{10.92}}
          \frac{1}{\alpha(\xi-\alpha)}\bigg(1+\int_{2}^{10.92(\xi-\alpha)-1}\frac{\log(\beta-1)}{\beta}
          \mathrm{d}\beta\bigg)\mathrm{d}\alpha\Bigg).
\end{align}

\subsubsection{Upper Bound Estimate for $S_2$}

By the definition of $\varrho_2(a)$, we have
\begin{align*}
          S_2
 = & \,\, \sum_{\substack{a\in\mathscr{A}\\ (a,P(x^{1/10.92}))=1\\ (a,2)=1}}
          \sum_{\substack{a=p_1p_2p_3\\ x^{1/10.92}\leqslant p_1<x^{1/3.29}\leqslant p_2<p_3}}1
                \nonumber \\
 = & \,\, \sum_{x^{1/10.92}\leqslant p_1<x^{1/3.29}}\sum_{x^{1/3.29}\leqslant p_2<(x/p_1)^{1/2}}
          \sum_{\substack{p=p_1p_2p_3-2\\ p_2<p_3\leqslant x/(p_1p_2)\\ p_1p_2p_3=[k^{1/\gamma}]}}1.
\end{align*}
Define
\begin{align*}
         \mathscr{E}
 = & \,\,\Big\{e:\,e=p_1p_2,\,x^{1/10.92}\leqslant p_1<x^{1/3.29}\leqslant p_2<(x/p_1)^{1/2}\Big\},
                        \nonumber \\
         \mathscr{L}
 = & \,\,\Big\{\ell:\,\ell=ep-2,\,e\in\mathscr{E},\,ep\leqslant x\Big\},
                        \nonumber \\
         \mathscr{B}
 = & \,\,\Big\{\ell:\ell\in\mathscr{L},\, \ell+2=[k^{1/\gamma}]\,\,\textrm{for some}\,\,k\in\mathbb{N}^+\Big\}.
\end{align*}
Obviously, there holds
\begin{equation*}
 \big|\mathscr{E}\big|\leqslant\sum_{x^{1/10.92}\leqslant p_1<x^{1/3.29}}\bigg(\frac{x}{p_1}\bigg)^{1/2}
 \ll\frac{x^{1/2+1/6.58}}{\log x}<x^{0.652}.
\end{equation*}
In addition, for $e\in\mathscr{E}$, one has
\begin{equation*}
\begin{cases}
e=p_1p_2\geqslant x^{1/3.29+1/10.92}\geqslant x^{0.3955}, \\
e=p_1p_2<p_1(x/p_1)^{1/2}=x^{1/2}p_1^{1/2}<x^{1/2+1/6.58}<x^{0.652}.
\end{cases}
\end{equation*}
We claim that the number of the elements in $\mathscr{L}$, which are not exceeding $x^{0.3955}$, is less than
$x^{0.652}$. Virtually, for such $\ell\in\mathscr{L}$ satisfying $\ell\leqslant x^{0.3955}$, then there exists
$e\in\mathscr{E}$ and prime $p$ such that
\begin{equation}\label{L-num}
0<ep-2=\ell\leqslant x^{0.3955}.
\end{equation}
If we fix the element $e\in\mathscr{E}$, the prime $p$ which satisfies the inequality (\ref{L-num}) must be unique.
Otherwise, if there exist two primes $p$ and $p'$ with $p\not=p'$ such that
\begin{equation*}
0<ep-2\leqslant x^{0.3955}\qquad \textrm{and}\qquad 0<ep'-2\leqslant x^{0.3955},
\end{equation*}
then we deduce that
\begin{equation*}
x^{0.3955}>\big|(ep-2)-(ep'-2)\big|=e|p-p'|\geqslant2e\geqslant2x^{0.3955},
\end{equation*}
which is a contradiction. Consequently, we obtain
\begin{align}\label{S_2-upper-1}
S_2 \leqslant & \,\, S(\mathscr{B},\mathscr{P},x^{0.3955})+O(x^{0.652})
                       \nonumber \\
\leqslant & \,\, S(\mathscr{B},\mathscr{P},z)+O(x^{0.652})
\end{align}
holds for $z\leqslant x^{0.3955}$.  Let $\mathscr{B}_d=\big\{\ell: \ell\in\mathscr{B}, \ell\equiv0\!\pmod d\big\}$. Then it is easy to see that
\begin{equation*}
   \big|\mathscr{B}_d\big|=\frac{1}{\varphi(d)}\mathcal{X}+\mathscr{R}_d^{(1)}+\mathscr{R}_d^{(2)},
\end{equation*}
where
\begin{equation*}
  \mathcal{X}=\sum_{\ell\in\mathscr{L}}\big((\ell+3)^\gamma-(\ell+2)^\gamma\big),
\end{equation*}
\begin{equation}\label{R_d(1)-def}
  \mathscr{R}_d^{(1)}=\sum_{\substack{\ell\in\mathscr{L}\\ \ell\equiv0\!\!\!\!\!\pmod d}}
  \big((\ell+3)^\gamma-(\ell+2)^\gamma\big)
  -\frac{1}{\varphi(d)}\sum_{\substack{\ell\in\mathscr{L}}}\big((\ell+3)^\gamma-(\ell+2)^\gamma\big),
\end{equation}
\begin{equation}\label{R_d(2)-def}
  \mathscr{R}_d^{(2)}=\sum_{\substack{\ell\in\mathscr{L}\\ \ell\equiv0\!\!\!\!\!\pmod d}}
  \Big(\psi\big(-(\ell+3)^\gamma\big)-\psi\big(-(\ell+2)^\gamma\big)\Big).
\end{equation}
In order to apply (\ref{upper-sieve}) to give upper bound estimate for $S(\mathscr{B},\mathscr{P},z)$, we need to show that
\begin{equation}\label{S_2-err-condi}
  \sum_{\substack{d\leqslant x^\xi\\ (d,2)=1}}\Big|\mathscr{R}_d^{(i)}\Big|\ll \frac{x^\gamma}{(\log x)^{A}},
  \qquad i=1,2.
\end{equation}
We shall prove (\ref{S_2-err-condi}) by two following lemmas. For convenience, we put $D=x^\xi$.
\begin{lemma}\label{R_1-err-mean}
Let $\mathscr{R}_d^{(1)}$ be defined as in (\ref{R_d(1)-def}). Then we have
\begin{equation*}
 \sum_{\substack{d\leqslant D\\ (d,2)=1}}\Big|\mathscr{R}_d^{(1)}\Big|\ll \frac{x^\gamma}{(\log x)^{A}}.
\end{equation*}
\end{lemma}
\begin{proof}
By the definition of the set $\mathscr{L}$, we only need to show that
\begin{equation}\label{S_2-error-eqiv-1}
 \sum_{\substack{d\leqslant D\\ (d,2)=1}}\Bigg|\sum_{e\in\mathscr{E}}\Bigg(\sum_{\substack{p<x/e\\ ep\equiv2\!\!\!\!\!\pmod d}}\big((ep+1)^\gamma-(ep)^\gamma\big)
 -\frac{1}{\varphi(d)}\sum_{p<x/e}\big((ep+1)^\gamma-(ep)^\gamma\big)\Bigg)
 \Bigg|\ll \frac{x^\gamma}{(\log x)^{A}}.
\end{equation}
On the left--hand side of (\ref{S_2-error-eqiv-1}), for the sum over $e\in\mathscr{E}$, we decompose the sum into
two parts according to whether $(e,d)=1$ or $(e,d)>1$. If $(e,d)>1$, by noting the fact that $(d,2)=1$ , we always
have $ep\not\equiv2\!\pmod d$ for all $p<x/e$. Thus the first inner sum in the brackets on the left--hand side of
(\ref{S_2-error-eqiv-1}) vanishes in this case. Hence the left--hand side of (\ref{S_2-error-eqiv-1}) is
$\leqslant \mathcal{R}_1+\mathcal{R}_2$, where
\begin{equation*}
\mathcal{R}_1=\sum_{\substack{d\leqslant D\\ (d,2)=1}}\Bigg|\sum_{\substack{e\in\mathscr{E}\\ (e,d)=1}}
\Bigg(\sum_{\substack{p<x/e\\ ep\equiv2\!\!\!\!\!\pmod d}}\big((ep+1)^\gamma-(ep)^\gamma\big)
 -\frac{1}{\varphi(d)}\sum_{p<x/e}\big((ep+1)^\gamma-(ep)^\gamma\big)\Bigg)\Bigg|,
\end{equation*}
and
\begin{equation*}
\mathcal{R}_2=\sum_{\substack{d\leqslant D\\ (d,2)=1}}\frac{1}{\varphi(d)}
\Bigg|\sum_{\substack{e\in\mathscr{E}\\ (e,d)>1}}\sum_{p<x/e}\big((ep+1)^\gamma-(ep)^\gamma\big)\Bigg|.
\end{equation*}
In view of $x^{0.3955}\leqslant e<x^{0.652}$ for $e\in\mathscr{E}$, by setting $\mathds{1}_{\mathscr{E}}(a)$ the
characteristic function of $\mathscr{E}$, i.e.,
\begin{equation*}
\mathds{1}_{\mathscr{E}}(a)=
\begin{cases}
1, & \textrm{if}\,a\in\mathscr{E}, \\
0, & \textrm{if}\,a\not\in\mathscr{E}.
\end{cases}
\end{equation*}
we can write
\begin{equation*}
\mathcal{R}_1=\sum_{\substack{d\leqslant D\\ (d,2)=1}}
\Bigg|\sum_{\substack{x^{0.3955}\leqslant a<x^{0.652}\\ (a,d)=1}}\!\!\!\mathds{1}_{\mathscr{E}}(a)
\Bigg(\sum_{\substack{p<x/a\\ ap\equiv2\!\!\!\!\!\pmod d}}\!\!\!\!\big((ap+1)^\gamma-(ap)^\gamma\big)
-\frac{1}{\varphi(d)}\sum_{p<x/a}\!\!\big((ap+1)^\gamma-(ap)^\gamma\big)\Bigg)\Bigg|,
\end{equation*}
and
\begin{equation*}
\mathcal{R}_2=\sum_{\substack{d\leqslant D\\ (d,2)=1}}\frac{1}{\varphi(d)}
\sum_{\substack{x^{0.3955}\leqslant a<x^{0.652}\\ (a,d)\geqslant x^{1/10.92}}}\mathds{1}_{\mathscr{E}}(a)
\sum_{p<x/a}\big((ap+1)^\gamma-(ap)^\gamma\big).
\end{equation*}
In order to prove the result of $\mathcal{R}_1\ll x^\gamma(\log x)^{-A}$, by partial summation we can reduce the matters to considering the case where there exists no weight on the prime variable $p$, and only need to show that
\begin{equation*}
\sum_{\substack{d\leqslant D\\ (d,2)=1}}
\Bigg|\sum_{\substack{x^{0.3955}\leqslant a<x^{0.652}\\ (a,d)=1}}\!\!\!\mathds{1}_{\mathscr{E}}(a)
\Bigg(\sum_{\substack{p<x/a\\ ap\equiv2\!\!\!\!\!\pmod d}}1
-\frac{1}{\varphi(d)}\sum_{p<x/a}1\Bigg)\Bigg|\ll\frac{x}{(\log x)^A},
\end{equation*}
By the theorem of Pan and Ding, i.e. (\ref{Pan-Ding-mean}), the above estimate follows immediately. For the upper bound estimate of $\mathcal{R}_2$, one has
\begin{align*}
            \mathcal{R}_2
 \ll & \,\, \sum_{\substack{d\leqslant D\\ (d,2)=1}}\frac{1}{\varphi(d)}
            \sum_{\substack{x^{0.3955}\leqslant a<x^{0.652}\\ (a,d)\geqslant x^{1/10.92}}}
            \sum_{p<x/a}\gamma(ap)^{\gamma-1}
      \ll   \sum_{\substack{d\leqslant D\\ (d,2)=1}}\frac{1}{\varphi(d)}
            \sum_{\substack{x^{0.3955}\leqslant a<x^{0.652}\\ (a,d)\geqslant x^{1/10.92}}}
            a^{\gamma-1}\int_2^{\frac{x}{a}}\frac{\gamma t^{\gamma-1}}{\log t}\mathrm{d}t
                  \nonumber \\
 \ll & \,\, \frac{x^\gamma}{\log x}\sum_{d\leqslant D}\frac{1}{\varphi(d)}
            \sum_{\substack{x^{0.3955}\leqslant a<x^{0.652}\\ (a,d)\geqslant x^{1/10.92}}}\frac{1}{a}
                 \ll \frac{x^\gamma}{\log x}\sum_{d\leqslant D}\frac{1}{\varphi(d)}
            \sum_{\substack{m|d\\ m\geqslant x^{1/10.92}}}\frac{1}{m}\sum_{a_1\leqslant x^{0.652}/m}\frac{1}{a_1}
                  \nonumber \\
 \ll & \,\, \frac{x^\gamma}{\log x}\sum_{x^{1/10.92}\leqslant m\leqslant D}\frac{1}{m\varphi(m)}
            \sum_{d_1\leqslant D/m}\frac{1}{\varphi(d_1)}\sum_{a_1\leqslant x^{0.652}/m}\frac{1}{a_1}
               \ll  x^{\gamma-1/10.92}\log^2x.
\end{align*}
This completes the proof of Lemma \ref{R_1-err-mean}.
\end{proof}

\begin{lemma}\label{R_2-err-mean}
Let $\mathscr{R}_d^{(2)}$ be defined as in (\ref{R_d(2)-def}). Then we have
\begin{equation*}
 \sum_{\substack{d\leqslant D\\ (d,2)=1}}\Big|\mathscr{R}_d^{(2)}\Big|\ll \frac{x^\gamma}{(\log x)^{A}}.
\end{equation*}
\end{lemma}
\begin{proof}
By the definition of $\mathscr{R}_d^{(2)}$, it suffices to show that, for $X\leqslant x$, there holds
\begin{equation}\label{R_2-suffi-condi}
 \sum_{\substack{d\leqslant D\\ (d,2)=1}}\Bigg|
 \sum_{\substack{\ell\in\mathscr{L}\\ \ell\equiv0\!\!\!\!\!\pmod d\\ \ell\sim X}}
 \Big(\psi\big(-(\ell+3)^\gamma\big)-\psi\big(-(\ell+2)^\gamma\big)\Big)\Bigg|\ll \frac{x^\gamma}{(\log x)^{A}}.
\end{equation}
If $X\leqslant x^{1-\eta}$, then the left--hand side of (\ref{R_2-suffi-condi}) is
\begin{align*}
\ll & \,\, \sum_{\substack{d\leqslant D\\ (d,2)=1}}\Bigg|
           \sum_{\substack{\ell\sim X\\ \ell\equiv0\!\!\!\!\!\pmod d}}\!\!
           \big((\ell+3)^\gamma-(\ell+2)^\gamma\big)\Bigg|+\sum_{\substack{d\leqslant D\\ (d,2)=1}}
           \Bigg|\sum_{\substack{\ell\sim X\\ \ell\equiv0\!\!\!\!\!\pmod d}}\!\!\!
           \big(\big[-(\ell+2)^\gamma\big]-\big[-(\ell+3)^\gamma\big]\big)\Bigg|
                 \nonumber \\
\ll & \,\, \sum_{\ell\sim X}(\ell+3)^{\gamma-1}\tau(\ell)+\sum_{\substack{d\leqslant D\\ (d,2)=1}}
           \sum_{\substack{\ell\sim X\\ \ell\equiv0\!\!\!\!\!\pmod d\\ \ell+2=[k^{1/\gamma}]}}1
           \ll X^{\gamma+\eta/2}\ll x^{\gamma}(\log x)^{-A}.
\end{align*}
Now, we assume that $x^{1-\eta}<X\leqslant x$. By (\ref{psi-expan}), the contribution of the error term in   (\ref{psi-expan}) to the left--hand side of (\ref{R_2-suffi-condi}) is
\begin{equation}
\ll\sum_{\substack{d\leqslant D\\ (d,2)=1}}
\sum_{\substack{\ell\sim X\\ \ell\equiv 0\!\!\!\!\!\pmod d}}
\big(E((\ell+2)^\gamma,H)+E((\ell+3)^\gamma,H)\big)=E_1^*+E_2^*,
\end{equation}
say. The treatment of $E_1^*$ and $E_2^*$ is exactly the same as that of (\ref{E_1-upper}) by using Lemma \ref{ex-arith-pair} with exponent pair $(\frac{1}{2},\frac{1}{2})$, provided that (\ref{suffi-condi-1}) holds.
So we omit the process herein.

The contribution of the main term in (\ref{psi-expan}) to the left--hand side of (\ref{R_2-suffi-condi}) is
\begin{align*}
= & \,\, \sum_{\substack{d\leqslant D\\ (d,2)=1}}\Bigg|
         \sum_{\substack{\ell\in\mathscr{L}\\ \ell\equiv0\!\!\!\!\!\pmod d\\ \ell\sim X}}
         \sum_{0<h\leqslant H}\frac{e(-h(\ell+2)^\gamma)-e(-h(\ell+3)^\gamma)}{2\pi ih}\Bigg|
               \nonumber \\
\ll & \,\, \sum_{\substack{d\leqslant D\\ (d,2)=1}}\sum_{0<h\leqslant H}\frac{1}{h}\Bigg|
           \sum_{\substack{\ell\in\mathscr{L}\\ \ell\equiv0\!\!\!\!\!\pmod d\\ \ell\sim X}}
           \big(e(-h(\ell+2)^\gamma)-e(-h(\ell+3)^\gamma)\big)\Bigg|=:\Upsilon,
\end{align*}
say. Define
\begin{equation*}
\mathfrak{f}_h(\ell)=1-e\big(h((\ell+2)^\gamma-(\ell+3)^\gamma)\big).
\end{equation*}
It follows from partial summation that
\begin{align*}
          \Upsilon
 = & \,\, \sum_{\substack{d\leqslant D\\ (d,2)=1}}\sum_{0<h\leqslant H}\frac{1}{h}\Bigg|
          \sum_{\substack{\ell\in\mathscr{L}\\ \ell\equiv0\!\!\!\!\!\pmod d\\ \ell\sim X}}
          e\big(-h(\ell+2)^\gamma\big)\mathfrak{f}_h(\ell)\Bigg|
                \nonumber \\
 = & \,\, \sum_{\substack{d\leqslant D\\ (d,2)=1}}\sum_{0<h\leqslant H}\frac{1}{h}\Bigg|
          \int_X^{2X}\mathfrak{f}_h(u)\mathrm{d}\Bigg(
          \sum_{\substack{\ell\in\mathscr{L}\\ \ell\equiv0\!\!\!\!\!\pmod d\\ X<\ell\leqslant u}}
          e\big(-h(\ell+2)^\gamma\big)\Bigg)\Bigg|
                \nonumber \\
 \ll & \,\, \sum_{\substack{d\leqslant D\\ (d,2)=1}}\sum_{0<h\leqslant H}\frac{1}{h}
            \Bigg(\Big|\mathfrak{f}_h(2X)\Big|\Bigg|
            \sum_{\substack{\ell\in\mathscr{L}\\ \ell\equiv0\!\!\!\!\!\pmod d\\ \ell\sim X}}
            e\big(-h(\ell+2)^\gamma\big)\Bigg|
                 \nonumber \\
  & \,\, \qquad +\int_X^{2X}\Bigg|
         \sum_{\substack{\ell\in\mathscr{L}\\ \ell\equiv0\!\!\!\!\!\pmod d\\ X<\ell\leqslant u}}
          e\big(-h(\ell+2)^\gamma\big)\Bigg|\bigg|\frac{\partial\mathfrak{f}_h(u)}{\partial u}\bigg|
          \mathrm{d}u\Bigg)
                 \nonumber \\
 \ll & \,\, X^{\gamma-1}\times\sum_{\substack{d\leqslant D\\ (d,2)=1}}\sum_{0<h\leqslant H}
           \max_{X<u\leqslant2X}\Bigg|
           \sum_{\substack{\ell\in\mathscr{L}\\ \ell\equiv0\!\!\!\!\!\pmod d\\ X<\ell\leqslant u}}
           e\big(-h(\ell+2)^\gamma\big)\Bigg|,
\end{align*}
where we use the estimate
\begin{equation*}
\big|\mathfrak{f}_h(u)\big|\ll hu^{\gamma-1}\qquad \textrm{and}\qquad
\bigg|\frac{\partial\mathfrak{f}_h(u)}{\partial u}\bigg|\ll hu^{\gamma-2}.
\end{equation*}
Therefore, we obtain
\begin{align}\label{R_2-divi-condi}
            \sum_{\substack{d\leqslant D\\ (d,2)=1}}\Big|\mathscr{R}_d^{(2)}\Big|
 \ll & \,\, \frac{x^\gamma}{(\log x)^{A}}+\max_{\substack{x^{1-\eta}<X\leqslant x\\ X<u\leqslant2X}}X^{\gamma-1}
            \sum_{\substack{d\leqslant D\\ (d,2)=1}}\sum_{0<h\leqslant H}
            \sum_{\substack{\ell\in\mathscr{L}\\ \ell\equiv0\!\!\!\!\!\pmod d\\ X<\ell\leqslant u}}
            \delta(d,h) e\big(-h(\ell+2)^\gamma\big)
                  \nonumber \\
 \ll & \,\, \frac{x^\gamma}{(\log x)^{A}}
            +\max_{\substack{x^{1-\eta}<X\leqslant x\\ X<u\leqslant2X}}
            X^{\gamma-1}\Bigg|\sum_{\substack{\ell\in\mathscr{L}\\ X<\ell\leqslant u}}
            \sum_{0<h\leqslant H}\Theta_h(\ell)e\big(-h(\ell+2)^\gamma\big)\Bigg|,
\end{align}
where
\begin{equation*}
\Theta_h(\ell)=\sum_{\substack{d\leqslant D\\ \ell\equiv0\!\!\!\!\!\pmod d\\ (d,2)=1 }}\delta(d,h),
\qquad \textrm{and}\qquad \big|\delta(d,h)\big|=1.
\end{equation*}
For $\ell\in\mathscr{L}$ and $\ell\sim X>x^{1-\eta}$, we have $\ell+2=p_1p_2p_3\in[x^{1-\eta},x]$ with $p_1$ and $p_2$ satisfying
\begin{equation*}
 x^{1/10.92}\leqslant p_1<x^{1/3.29}\leqslant p_2<\bigg(\frac{x}{p_1}\bigg)^{1/2},
\end{equation*}
and we claim that there must be a subproduct of $p_1p_2p_3$ which lies in the interval
$[X^{\frac{45}{76}+\eta},X^{\gamma_0-\eta}]$, where
\begin{equation*}
\gamma_0=1-\frac{18}{2^{38}+17}.
\end{equation*}

Virtually, since $p_1\geqslant x^{1/10.92},p_2\geqslant x^{1/3.29}$ and $p_1p_2p_3\in[x^{1-\eta},x]$, we have
$p_i\leqslant X^{\gamma_0-\eta}$. If there exists some $p_i\in[X^{\frac{45}{76}+\eta},X^{\gamma_0-\eta}]$, then the conclusion follows.  If this case does not exist, we consider the product $p_1p_2$. Trivially, there holds
$p_1p_2<p_2^2<x/p_1\leqslant x^{1-\frac{1}{10.92}}\leqslant x^{\gamma_0-\eta}$. If $p_1p_2\in[X^{\frac{45}{76}+\eta},X^{\gamma_0-\eta}]$, then the conclusion follows. Otherwise, if $p_1p_2<X^{\frac{45}{76}+\eta}$, then $p_3=(\ell+2)(p_1p_2)^{-1}>X(X^{\frac{45}{76}+\eta})^{-1}=
X^{\frac{31}{76}-\eta}>x^{\frac{15}{38}}$, and hence $p_2p_3>x^{\frac{1}{3.29}+\frac{15}{38}}>X^{\frac{45}{76}+\eta}$. In addition, $p_2p_3=(\ell+2)p_1^{-1}\leqslant x^{1-\frac{1}{10.92}}\leqslant X^{\gamma_0-\eta}$. Above all cases, there must exist some subproduct of $p_1p_2p_3$ which lies in the interval $[X^{\frac{45}{76}+\eta},X^{\gamma_0-\eta}]$.

For $\gamma_0<\gamma<1$ and the definition of $\xi$, it is easy to see that
\begin{equation*}
  X^{\frac{61(1-\gamma)+5\xi}{4}+\eta}\leqslant X^{\frac{45}{76}+\eta}<X^{\gamma_0-\eta}\leqslant X^{\gamma-\eta},
\end{equation*}
which combined with (\ref{R_2-divi-condi}) and Lemma \ref{Type-II-es} with exponent pair $(\kappa,\ell)=A^3(\frac{1}{2},\frac{1}{2})=(\frac{1}{30},\frac{13}{15})$ yields
\begin{equation*}
\sum_{\substack{d\leqslant D\\ (d,2)=1}}\Big|\mathscr{R}_d^{(2)}\Big|\ll \frac{x^\gamma}{(\log x)^{A}}.
\end{equation*}
This completes the proof of  Lemma \ref{R_2-err-mean}.
\end{proof}

Taking $z=x^{\xi/3}$ in (\ref{S_2-upper-1}), it follows from (\ref{upper-sieve}), Lemma \ref{R_1-err-mean} and
Lemma \ref{R_2-err-mean} that
\begin{equation}\label{S_B-upp}
S(\mathscr{B},\mathscr{P},x^{\xi/3})\leqslant\mathcal{X}V(x^{\xi/3})F(3)(1+o(1)).
\end{equation}
Next, we compute the quantity $\mathcal{X}$ definitely. We write
\begin{equation}\label{S_2-X-rewri}
\mathcal{X}=\sum_{\ell\in\mathscr{L}}\gamma\ell^{\gamma-1}+\sum_{\ell\in\mathscr{L}}
\Big((\ell+3)^\gamma-(\ell+2)^\gamma-\gamma\ell^{\gamma-1}\Big).
\end{equation}
For the second term in (\ref{S_2-X-rewri}), one has
\begin{align}\label{S_2-X-err-upp}
    & \,\, \sum_{\ell\in\mathscr{L}}\Big((\ell+3)^\gamma-(\ell+2)^\gamma-\gamma\ell^{\gamma-1}\Big)
           \ll\sum_{\ell\in\mathscr{L}}\ell^{\gamma-2}
               \nonumber \\
 \ll & \,\, \sum_{\substack{x^{1/10.92}\leqslant p_1<x^{1/3.29}\leqslant p_2<(x/p_1)^{1/2}\\ p_3<x/(p_1p_2)}}
            (p_1p_2p_3-2)^{\gamma-2}
               \nonumber \\
 \ll & \,\, \sum_{\substack{x^{1/10.92}\leqslant p_1<x^{1/3.29}\leqslant p_2<(x/p_1)^{1/2}}}
            (p_1p_2)^{\gamma-2}\sum_{p_3<x/(p_1p_2)}p_3^{\gamma-2}
               \nonumber \\
 \ll & \,\, \bigg(\sum_{x^{1/10.92}\leqslant p\leqslant x}p^{\gamma-2}\bigg)^2
            \ll\bigg(\sum_{x^{1/10.92}\leqslant m\leqslant x}m^{\gamma-2}\bigg)^2
            \ll x^{2(\gamma-1)/10.92}=o(1).
\end{align}
For the first term in (\ref{S_2-X-rewri}), we have
\begin{align}\label{S_2-X-main}
   & \,\, \sum_{\ell\in\mathscr{L}}\gamma\ell^{\gamma-1}
          =\gamma\sum_{e\in\mathscr{E}}\sum_{p_3<x/e}(ep_3-2)^{\gamma-1}
                \nonumber \\
 = & \,\, \gamma \sum_{x^{1/10.92}\leqslant p_1< x^{1/3.29}}\sum_{x^{1/3.29}\leqslant p_2<(x/p_1)^{1/2}}
          \sum_{2<p_3<x/(p_1p_2)}(p_1p_2p_3-2)^{\gamma-1}
                \nonumber \\
 = & \,\, \gamma\big(1+o(1)\big)\int_{x^{1/10.92}}^{x^{1/3.29}}\int_{x^{1/3.29}}^{(\frac{x}{t_1})^{1/2}}
          \int_{2}^{\frac{x}{t_1t_2}}
          \frac{(t_1t_2t_3-2)^{\gamma-1}}{(\log t_1)(\log t_2)(\log t_3)}\mathrm{d}t_3\mathrm{d}t_2\mathrm{d}t_1
                 \nonumber \\
 = & \,\, \gamma\big(1+o(1)\big)\int_{\frac{1}{10.92}}^{\frac{1}{3.29}}\frac{\mathrm{d}\alpha_1}{\alpha_1}
          \int_{\frac{1}{3.29}}^{\frac{1-\alpha_1}{2}}\frac{\mathrm{d}\alpha_2}{\alpha_2}
          \int_{\frac{\log 2}{\log x}}^{1-\alpha_1-\alpha_2}
          \frac{(x^{\alpha_1+\alpha_2+\alpha_3}-2)^{\gamma-1}}{\alpha_3}
          \cdot x^{\alpha_1+\alpha_2+\alpha_3}\mathrm{d}\alpha_3
                 \nonumber \\
 = & \,\, \gamma\big(1+o(1)\big)\int_{\frac{1}{10.92}}^{\frac{1}{3.29}}\frac{\mathrm{d}\alpha_1}{\alpha_1}
          \int_{\frac{1}{3.29}}^{\frac{1-\alpha_1}{2}}\frac{\mathrm{d}\alpha_2}{\alpha_2}
          \int_{\frac{\log 2}{\log x}}^{1-\alpha_1-\alpha_2}
          \frac{x^{(\alpha_1+\alpha_2+\alpha_3)\gamma}}{\alpha_3}\mathrm{d}\alpha_3.
\end{align}
For the innermost integral in (\ref{S_2-X-main}), we have
\begin{align}\label{inner-expli}
       \int_{\frac{\log 2}{\log x}}^{1-\alpha_1-\alpha_2}
       \frac{x^{(\alpha_1+\alpha_2+\alpha_3)\gamma}}{\alpha_3}\mathrm{d}\alpha_3
   = & \,\, \frac{1}{\gamma \log x}\int_{\frac{\log 2}{\log x}}^{1-\alpha_1-\alpha_2}
            \frac{1}{\alpha_3}\mathrm{d}x^{(\alpha_1+\alpha_2+\alpha_3)\gamma}
                  \nonumber \\
   = & \,\, \frac{1}{\gamma \log x}\bigg(\frac{x^\gamma}{1-\alpha_1-\alpha_2}
            +O\bigg(\frac{x^\gamma}{\log x}\bigg)\bigg)
                   \nonumber \\
   = & \,\, \frac{1}{1-\alpha_1-\alpha_2}\cdot\frac{x^\gamma}{\gamma\log x}(1+o(1)).
\end{align}
From (\ref{S_2-X-rewri}), (\ref{S_2-X-err-upp}), (\ref{S_2-X-main}) and (\ref{inner-expli}), we obtain
\begin{equation}\label{S_2-X-compu-num}
 \mathcal{X}=\frac{x^\gamma(1+o(1))}{\log x}\int_{\frac{1}{10.92}}^{\frac{1}{3.29}}
             \frac{\mathrm{d}\alpha_1}{\alpha_1}
          \int_{\frac{1}{3.29}}^{\frac{1-\alpha_1}{2}}\frac{\mathrm{d}\alpha_2}{\alpha_2(1-\alpha_1-\alpha_2)}.
\end{equation}
Combining (\ref{S_2-upper-1}), (\ref{S_B-upp}), (\ref{S_2-X-compu-num}), we derive that
\begin{equation}\label{S_2-upp-final}
S_2\leqslant\frac{4\mathfrak{S}}{\xi}\Bigg(\int_{\frac{1}{10.92}}^{\frac{1}{3.29}}
\frac{\mathrm{d}\alpha_1}{\alpha_1}\int_{\frac{1}{3.29}}^{\frac{1-\alpha_1}{2}}
\frac{\mathrm{d}\alpha_2}{\alpha_2(1-\alpha_1-\alpha_2)}\Bigg)\cdot\frac{x^\gamma}{\log^2x}(1+o(1)).
\end{equation}

\subsubsection{Upper Bound Estimate for $S_3$}

By the definition of $\varrho_3(a)$, we have
\begin{align*}
          S_3
 = & \,\, \sum_{\substack{a\in\mathscr{A}\\ (a,P(x^{1/10.92}))=1\\ (a,2)=1}}
          \sum_{\substack{a=p_1p_2p_3\\ x^{1/3.29}\leqslant p_1<p_2<p_3}}1
                \nonumber \\
 = & \,\, \sum_{x^{1/3.29}\leqslant p_1<x^{1/3}}\sum_{p_1<p_2<(x/p_1)^{1/2}}
          \sum_{\substack{p=p_1p_2p_3-2\\ p_2<p_3\leqslant x/(p_1p_2)\\ p_1p_2p_3=[k^{1/\gamma}]}}1.
\end{align*}
Define
\begin{align*}
         \mathscr{E}^*
 = & \,\,\Big\{e:\,e=p_1p_2,\,x^{1/3.29}\leqslant p_1<p_2<(x/p_1)^{1/2}\Big\},
                        \nonumber \\
         \mathscr{L}^*
 = & \,\,\Big\{\ell:\,\ell=ep-2,\,e\in\mathscr{E}^*,\,ep\leqslant x\Big\},
                        \nonumber \\
         \mathscr{B}^*
 = & \,\,\Big\{\ell:\ell\in\mathscr{L}^*,\, \ell+2=[k^{1/\gamma}]\,\,\textrm{for some}\,\,k\in\mathbb{N}^+\Big\}.
\end{align*}
Obviously, we have
\begin{equation*}
 \big|\mathscr{E}^*\big|\leqslant\sum_{x^{1/3.29}\leqslant p_1<x^{1/3}}\bigg(\frac{x}{p_1}\bigg)^{1/2}
 \ll\frac{x^{2/3}}{\log x}<x^{2/3}.
\end{equation*}
In addition, for $e\in\mathscr{E}^*$, one has
\begin{equation*}
\begin{cases}
e=p_1p_2>p_1^2\geqslant x^{0.607}, \\
e=p_1p_2<p_1(x/p_1)^{1/2}=x^{1/2}p_1^{1/2}<x^{2/3}.
\end{cases}
\end{equation*}
Similar to the illustration in (\ref{L-num}), it is easy to see that the number of the elements in $\mathscr{L}^*$, which are not exceeding $x^{0.607}$, is less than $x^{2/3}$. Accordingly, we obtain
\begin{align}\label{S_3-upper-1}
S_3 \leqslant & \,\, S(\mathscr{B}^*,\mathscr{P},x^{0.607})+O(x^{2/3})
                       \nonumber \\
\leqslant & \,\, S(\mathscr{B}^*,\mathscr{P},z)+O(x^{2/3})
\end{align}
holds for $z\leqslant x^{0.607}$.  Let
$\mathscr{B}^*_d=\big\{\ell: \ell\in\mathscr{B}^*, \ell\equiv0\!\pmod d\big\}$. Then it is easy to see that
\begin{equation*}
   \big|\mathscr{B}^*_d\big|=\frac{1}{\varphi(d)}\mathcal{X}^*+\mathscr{R}_d^{*(1)}+\mathscr{R}_d^{*(2)},
\end{equation*}
where
\begin{equation*}
  \mathcal{X}^*=\sum_{\ell\in\mathscr{L}^*}\big((\ell+3)^\gamma-(\ell+2)^\gamma\big),
\end{equation*}
\begin{equation}\label{R_d^*(1)-def}
  \mathscr{R}_d^{*(1)}=\sum_{\substack{\ell\in\mathscr{L}^*\\ \ell\equiv0\!\!\!\!\!\pmod d}}
  \big((\ell+3)^\gamma-(\ell+2)^\gamma\big)
  -\frac{1}{\varphi(d)}\sum_{\substack{\ell\in\mathscr{L}^*}}\big((\ell+3)^\gamma-(\ell+2)^\gamma\big),
\end{equation}
\begin{equation}\label{R_d^*(2)-def}
  \mathscr{R}_d^{*(2)}=\sum_{\substack{\ell\in\mathscr{L}^*\\ \ell\equiv0\!\!\!\!\!\pmod d}}
  \Big(\psi\big(-(\ell+3)^\gamma\big)-\psi\big(-(\ell+2)^\gamma\big)\Big).
\end{equation}
Moreover, one can follow the same process of Lemma \ref{R_1-err-mean} and Lemma \ref{R_2-err-mean}
to establish that
\begin{equation}\label{S_3-err-condi}
  \sum_{\substack{d\leqslant x^\xi\\ (d,2)=1}}\Big|\mathscr{R}_d^{*(i)}\Big|\ll \frac{x^\gamma}{(\log x)^{A}},
  \qquad i=1,2.
\end{equation}
Now, we describe the quantity $\mathcal{X}^*$ with explicit coefficient and definite magnitude.
We write
\begin{equation}\label{S_3-X-rewri}
\mathcal{X}^*=\sum_{\ell\in\mathscr{L}^*}\gamma\ell^{\gamma-1}+\sum_{\ell\in\mathscr{L}^*}
\Big((\ell+3)^\gamma-(\ell+2)^\gamma-\gamma\ell^{\gamma-1}\Big).
\end{equation}
For the second term in (\ref{S_3-X-rewri}), one has
\begin{align}\label{S_3-X-err-upp}
    & \,\, \sum_{\ell\in\mathscr{L}^*}\Big((\ell+3)^\gamma-(\ell+2)^\gamma-\gamma\ell^{\gamma-1}\Big)
           \ll\sum_{\ell\in\mathscr{L}^*}\ell^{\gamma-2}
               \nonumber \\
 \ll & \,\, \sum_{\substack{x^{1/3.29}\leqslant p_1<p_2<(x/p_1)^{1/2}\\ p_3<x/(p_1p_2)}}
            (p_1p_2p_3-2)^{\gamma-2}
              \ll \sum_{\substack{x^{1/3.29}\leqslant p_1<p_2<(x/p_1)^{1/2}}}
            (p_1p_2)^{\gamma-2}\sum_{p_3<x/(p_1p_2)}p_3^{\gamma-2}
               \nonumber \\
 \ll & \,\, \bigg(\sum_{x^{1/3.29}\leqslant p\leqslant x}p^{\gamma-2}\bigg)^2
            \ll\bigg(\sum_{x^{1/3.29}\leqslant m\leqslant x}m^{\gamma-2}\bigg)^2
            \ll x^{2(\gamma-1)/3.29}=o(1).
\end{align}
For the first term in (\ref{S_3-X-rewri}), we have
\begin{align}\label{S_3-X-main}
   & \,\, \sum_{\ell\in\mathscr{L}^*}\gamma\ell^{\gamma-1}
          =\gamma\sum_{e\in\mathscr{E}^*}\sum_{p_3<x/e}(ep_3-2)^{\gamma-1}
                \nonumber \\
 = & \,\, \gamma\sum_{x^{1/3.29}\leqslant p_1< x^{1/3}}\sum_{p_1<p_2<(x/p_1)^{1/2}}
          \sum_{2<p_3<x/(p_1p_2)}(p_1p_2p_3-2)^{\gamma-1}
                \nonumber \\
 = & \,\, \gamma\big(1+o(1)\big)\int_{x^{1/3.29}}^{x^{1/3}}\int_{t_1}^{(\frac{x}{t_1})^{1/2}}
          \int_{2}^{\frac{x}{t_1t_2}}
          \frac{(t_1t_2t_3-2)^{\gamma-1}}{(\log t_1)(\log t_2)(\log t_3)}\mathrm{d}t_3\mathrm{d}t_2\mathrm{d}t_1
                 \nonumber \\
 = & \,\, \gamma\big(1+o(1)\big)\int_{\frac{1}{3.29}}^{\frac{1}{3}}\frac{\mathrm{d}\alpha_1}{\alpha_1}
          \int_{\alpha_1}^{\frac{1-\alpha_1}{2}}\frac{\mathrm{d}\alpha_2}{\alpha_2}
          \int_{\frac{\log 2}{\log x}}^{1-\alpha_1-\alpha_2}
          \frac{(x^{\alpha_1+\alpha_2+\alpha_3}-2)^{\gamma-1}}{\alpha_3}
          \cdot x^{\alpha_1+\alpha_2+\alpha_3}\mathrm{d}\alpha_3
                 \nonumber \\
 = & \,\, \gamma\big(1+o(1)\big)\int_{\frac{1}{3.29}}^{\frac{1}{3}}\frac{\mathrm{d}\alpha_1}{\alpha_1}
          \int_{\alpha_1}^{\frac{1-\alpha_1}{2}}\frac{\mathrm{d}\alpha_2}{\alpha_2}
          \int_{\frac{\log 2}{\log x}}^{1-\alpha_1-\alpha_2}
          \frac{x^{(\alpha_1+\alpha_2+\alpha_3)\gamma}}{\alpha_3}\mathrm{d}\alpha_3
                 \nonumber \\
 = & \,\, \frac{x^\gamma(1+o(1))}{\log x}\int_{\frac{1}{3.29}}^{\frac{1}{3}}
          \frac{\mathrm{d}\alpha_1}{\alpha_1}
          \int_{\alpha_1}^{\frac{1-\alpha_1}{2}}\frac{\mathrm{d}\alpha_2}{\alpha_2(1-\alpha_1-\alpha_2)},
\end{align}
which combined with (\ref{S_3-X-err-upp}) gives
\begin{equation}\label{S_3-X-compu-num}
 \mathcal{X}^*=\frac{x^\gamma(1+o(1))}{\log x}\int_{\frac{1}{3.29}}^{\frac{1}{3}}
             \frac{\mathrm{d}\alpha_1}{\alpha_1}
          \int_{\alpha_1}^{\frac{1-\alpha_1}{2}}\frac{\mathrm{d}\alpha_2}{\alpha_2(1-\alpha_1-\alpha_2)}.
\end{equation}
Taking $z=x^{\xi/3}$ in (\ref{S_3-upper-1}), it follows from (\ref{upper-sieve}), (\ref{S_3-err-condi}) and
(\ref{S_3-X-compu-num}) that
\begin{align}\label{S_3-upp-final}
                  S_3
 \leqslant & \,\, S(\mathscr{B}^*,\mathscr{P},x^{\xi/3})+O(x^{2/3})
 \leqslant  \mathcal{X}^*V(x^{\xi/3})F(3)(1+o(1))
                  \nonumber \\
 \leqslant & \,\, \frac{4\mathfrak{S}}{\xi}\Bigg(\int_{\frac{1}{3.29}}^{\frac{1}{3}}
 \frac{\mathrm{d}\alpha_1}{\alpha_1}\int_{\alpha_1}^{\frac{1-\alpha_1}{2}}
 \frac{\mathrm{d}\alpha_2}{\alpha_2(1-\alpha_1-\alpha_2)}\Bigg)\cdot\frac{x^\gamma}{\log^2x}(1+o(1)).
\end{align}

\subsubsection{Proof of Theorem \ref{Theorem-2}}

It follows from (\ref{lower-tran-2}), (\ref{S-lower-final}), (\ref{S_1-upper-final}), (\ref{S_2-upp-final})
and (\ref{S_3-upp-final}) that
\begin{align*}
          & \,\, \big|\big\{a:a\in\mathscr{A},\,\,a=\mathcal{P}_2\big\}\big|
                      \nonumber \\
\geqslant & \,\,\frac{\mathfrak{S}x^\gamma(1+o(1))}{\log^2x}\Bigg\{\frac{4}{\xi}\bigg(\log(10.92\xi-1)
        +\int_3^{10.92\xi-1}\frac{\mathrm{d}t_1}{t_1}\int_2^{t_1-1}\frac{\log(t_2-1)}{t_2}\mathrm{d}t_2\bigg)
                      \nonumber \\
  & \,\, -2\Bigg(\int_{\xi-\frac{3}{10.92}}^{\frac{1}{3.29}}
          \frac{\mathrm{d}\alpha}{\alpha(\xi-\alpha)}
                     +\int_{\frac{1}{10.92}}^{\xi-\frac{3}{10.92}}
          \frac{1}{\alpha(\xi-\alpha)}\bigg(1+\int_{2}^{10.92(\xi-\alpha)-1}\frac{\log(\beta-1)}{\beta}
          \mathrm{d}\beta\bigg)\mathrm{d}\alpha\Bigg)
                      \nonumber \\
  & \,\, -\frac{2}{\xi}\Bigg(\int_{\frac{1}{10.92}}^{\frac{1}{3.29}}
         \frac{\mathrm{d}\alpha_1}{\alpha_1}\int_{\frac{1}{3.29}}^{\frac{1-\alpha_1}{2}}
         \frac{\mathrm{d}\alpha_2}{\alpha_2(1-\alpha_1-\alpha_2)}\Bigg)
      -\frac{4}{\xi}\Bigg(\int_{\frac{1}{3.29}}^{\frac{1}{3}}
      \frac{\mathrm{d}\alpha_1}{\alpha_1}\int_{\alpha_1}^{\frac{1-\alpha_1}{2}}
      \frac{\mathrm{d}\alpha_2}{\alpha_2(1-\alpha_1-\alpha_2)}\Bigg)\Bigg\}.
\end{align*}
By simple numerical calculations, it is easy to see that, for
\begin{equation*}
1-\frac{0.03208}{2^{38}+17}<\gamma<1,
\end{equation*}
there holds $\xi\geqslant0.47284$ and the number in the above brackets $\{\cdot\}$ is
$\geqslant0.000109508$, which completes the proof of Theorem \ref{Theorem-2}.

\section*{Acknowledgement}

%\noindent
The authors would like to appreciate the referee for his/her patience in refereeing this paper.
This work is supported by the National Natural Science Foundation of China (Grant Nos. 11901566, 12001047, 11971476, 12071238), and the Fundamental Research Funds for the Central Universities (Grant No. 2022YQLX05).

\end{document}